%% file: main.tex
\documentclass[twoside,11pt]{article}

\def\vgap{\vspace*{.1in}}

\usepackage{color}
\usepackage{bm,multirow}
\usepackage{enumerate,mathtools}
\usepackage{graphicx}
\usepackage{epstopdf,amsmath}
\usepackage{authblk}
\usepackage{amsfonts}
\usepackage{ amssymb }

\usepackage{makecell}

\allowdisplaybreaks

\usepackage{amsmath}    
\usepackage[table]{xcolor}
\usepackage[T1]{fontenc}
\usepackage[utf8]{inputenc}

\usepackage{amsmath,color}
\usepackage{algorithm}

\usepackage{enumerate}
\usepackage{algpseudocode}
\newenvironment{proof}{\par\noindent{\bf Proof\ }}{\hfill\BlackBox\\[2mm]}
\usepackage{graphicx}
\usepackage{epstopdf}
\usepackage{subfigure}
\usepackage{amsfonts}
\usepackage{enumitem}

\allowdisplaybreaks

\usepackage{amsmath}    
\usepackage[table]{xcolor}
\usepackage[T1]{fontenc}
\usepackage[utf8]{inputenc}

\usepackage{epstopdf}
\usepackage{amsmath}
\usepackage{epsf}
\usepackage{graphicx}
\usepackage{mathtools} 

\DeclareMathOperator*{\argmax}{argmax}
\DeclareMathOperator*{\argmin}{argmin}

\newcommand{\abs}[1]{\left\lvert #1 \right\rvert}
\newcommand{\norm}[1]{\left\lVert #1 \right\rVert}
\newcommand{\ceil}[1]{\lceil #1 \rceil}

\newcommand{\expec}[1]{\mathbf{E}\left[ #1 \right] }
\newcommand\numberthis{\addtocounter{equation}{1}\tag{\theequation}}  

\def\fn[#1]#2{{f_{#1}\left(x_{#2}\right)}}

\newcommand{\kb}[1]{{\color{red}\bf[KB: #1]}}

\newcommand{\ar}[1]{{\color{teal}\bf[AR: #1]}}

\newtheorem{definition}{Definition}[section]
\newtheorem{lemma}{Lemma}[section]
\newtheorem{theorem}{Theorem}[section]
\newtheorem{proposition}{Proposition}[section]
\newtheorem{assumption}{Assumption}[section]
\newtheorem{remark}{Remark}

\newcommand{\calL}{{\mathcal{L}}}
\newcommand{\calJ}{{\mathcal{J}}}
\newcommand{\calF}{{\mathcal{F}}}
\newcommand{\order}{\ensuremath{\mathcal{O}}}
\newcommand{\calX}{{\mathcal{X}}}
\newcommand{\calY}{{\mathcal{Y}}}
\newcommand{\inner}[1]{ \left\langle {#1} \right\rangle }
\newcommand{\derivx}[1]{\nabla_x f({#1})}
\newcommand{\derivy}[1]{\nabla_y f({#1})}
\newcommand{\wh}{\widehat}

\def\E{{\bf E}}

\def\half{{\textstyle{\frac{1}{2}}}}

\def\cX{{\cal X}}
\def\cY{{\cal Y}}

\def\tx{{\tilde x}}
\def\ty{{\tilde y}}

\def\hx{{\hat x}}
\def\hy{{\hat y}}
\newtheorem{cor}{Corollary}

\newcommand{\dd}{\xi}  
\newcommand{\pp}{x} 

\title{Online and Bandit Algorithms for Nonstationary Stochastic Saddle-Point Optimization}
\author[1]{Abhishek Roy\thanks{abroy@ucdavis.edu}}
\affil[1]{Department of Electrical and Computer Engineering, University of California, Davis}
\author[2]{Yifang Chen\thanks{yifang@usc.edu. Work done during an internship visit to UC Davis}}
\affil[2]{Department of Computer Science, University of Southern California}
\author[3]{Krishnakumar Balasubramanian\thanks{kbala@ucdavis.edu}}
\affil[3]{Department of Statistics, University of California, Davis}
\author[4]{Prasant Mohapatra\thanks{pmohapatra@ucdavis.edu}}
\affil[4]{Department of Computer Science, University of California, Davis}

\begin{document}
\maketitle
\begin{abstract}
Saddle-point optimization problems are an important class of optimization problems with applications to game theory, multi-agent reinforcement learning and machine learning. A majority of the rich literature available for saddle-point optimization has focused on the offline setting. In this paper, we study nonstationary versions of stochastic, smooth, strongly-convex and strongly-concave saddle-point optimization problem, in both online (or first-order) and multi-point bandit (or zeroth-order) settings. We first propose natural notions of regret for such nonstationary saddle-point optimization problems. We then analyze extragradient and Frank-Wolfe algorithms, for the unconstrained and constrained settings respectively, for the above class of nonstationary saddle-point optimization problems. We establish sub-linear regret bounds on the proposed notions of regret in both the online and bandit setting.
\end{abstract}
\newpage
\input{intro}
\input{staticregret}
\input{dynamicregret.tex}

\input{meta}
\section{Discussion}
In this work, we proposed and analyzed algorithms for sequential decision making problems that could be naturally formulated as nonstationary strongly-convex and strongly-concave saddle point optimization problems. We considered both static and dynamic notions of regret. We analyzed online and bandit versions of iterative algorithms like extragradient method and Frank-Wolfe method with respect to the proposed notions of regret, establishing sublinear regret bounds.

For future work, it is interesting to establish parameter-free versions of our algorithms. Furthermore it is interesting to explore nonconvex and nonconcave nonstationary saddle point optimization problems in the online and bandit settings, by extending appropriately, the definitions of regret proposed in~\cite{roy2019multi} for the \texttt{argmin}-type online nonconvex problems recently. Furthermore, recently the problem of convex body chasing, considered in~\cite{friedman1993convex} initially, has regained significant attention; see for example~\cite{bubeck2019competitively,sellke2019chasing,argue2019chasing}. It is intriguing to precisely formulate saddle-point versions of convex bodies chasing problem and explore algorithms for it.

\input{appendix}

\bibliographystyle{alpha}
\bibliography{ns_saddle}
\newpage

\end{document}

%% file: intro.tex
\section{Introduction}  
Sequential decision making problems are usually formulated as solving standard \texttt{argmin}-type convex optimization problems in an online fashion. Specifically, consider a sequence of $d$-dimensional real-valued functions $\{f_t(x) \}_{t=1}^T$ such that $x^*_t = \argmin_{ x \in \mathcal{X}}~f_t\left(x\right):= \E_\xi [F_t (x,\xi)]$, for some closed convex set $\mathcal{X} \subset \mathbb{R}^d$. Due to the sequential or online nature of the problem, in each round $t$, the decision-maker picks a decision $x_t$ and observes the (stochastic) loss suffered, $F_t(x_t, \xi_t)$, as a consequence of picking that decision. The goal in this setting to algorithmically produce a sequence of decisions $x_t$, based on the feedback received, such that the decisions compares favorably against an appropriately defined notion of regret, which is based on a certain oracle decision rule. In the so-called static setting, the oracle decision rule compared against is $\bar{x}^*:=\argmin_{ x \in \mathcal{X}}~\sum_{t=1}^T f_t\left(x\right)$ and the regret is defined $\mathcal{R} = \sum_{t=1}^T f_t(x_t) - \sum_{t=1}^T f_t( \bar{x}^*)$. In the so-called dynamic setting, it is typical to consider the following stronger notion $\mathcal{R} = \sum_{t=1}^T f_t(x_t) - \sum_{t=1}^T f_t( x_t^*)$. Algorithms developed for the above problems are typically called as online convex optimization algorithms in the literature. More recently, extension to structured non-convex functions (for example, sub-modular or quasi-convex functions) and general non-convex functions (with appropriately defined notions of local regrets) have been considered in the literature~\cite{hazan2017efficient, gao2018online, roy2019multi}. A common theme in all the above works is that they are based on the standard \texttt{argmin}-type optimization formulations. 

In this work, we study sequential decision making problems that could naturally be modeled as solving saddle-point optimization problems in an online fashion. Consider a sequence of functions $\{ f_t\left(x,y\right)\}_{t=1}^T$ and a corresponding sequence of points $\{(x^*_t,y^*_t)\}_{t=1}^T$ defined as
\begin{align}\label{eq:nssaddle}
(x^*_t,y^*_t) = \underset{ x \in \mathcal{X}}{\argmin}~\underset{y \in \mathcal{Y}}{\argmax}~f_t\left(x,y\right):= \E_\xi [F_t (x,y,\xi)].
\end{align}
Here, each function $f_t: \mathbb{R}^{d_X + d_Y} \to \mathbb{R}$ and the sets $\calX \subset \mathbb{R}^{d_X}$, $\calY \subset \mathbb{R}^{d_Y}$ are closed and convex. For the case of $T=1$, the above problem is called as offline saddle-point optimization problem in the literature. To solve such offline saddle point optimization problems, iterative algorithms like Gradient Descent Ascent (GDA) and variants, and Frank-Wolfe algorithms have been developed; see, for example,~\cite{korpelevich1976extragradient, rockafellar1976monotone,guler1991convergence, nemirovski2004prox,nedic2009subgradient, gidel2017frank} for a partial overview of such methods. We consider the online (and nonstationary) variant of the saddle-point optimization problem, where $T >1$ and in each iteration of an algorithm, the function being optimized changes. This is a natural extension of the standard online \texttt{argmin}-type optimization problems to the saddle-point optimization setting and was recently also considered in~\cite{2018arXiv180608301R}. In each round $t$, the decision-maker then picks actions $(x_t,y_t)$ and observes potentially noisy function evaluation feedback of the form $F_t(x_t,y_t,\xi_t)$ (or in some cases, also feedbacks of the form $F_i(x_t,y_t, \xi_{t,i})$ for $1\leq i \leq t$) as a consequence of picking the actions. This setting is called as the bandit setting. In some cases, noisy gradient or higher-order derivative information regarding the functions $F_i$, $1\leq i\leq t$ maybe obtained as well. This setting is typically referred to as the online setting. The goal in either setting, is to obtain a sequence of decisions $(x_t,y_t)$, based on the feedbacks obtained, so that the decisions compare favorably against an appropriately defined notion of regret. 

An immediate challenge that arises when trying to formulate the above problem is: How to define a meaningful notion of static and dynamic regret for online saddle-point optimization problems for which efficient algorithms could be designed? In the case of offline convex-concave saddle-point optimization, i.e., when the function $f(\cdot,y)$ is convex for all $y$ and the function $f(x,\cdot)$ is concave for all $x$, the so-called Nash equilibrium solution is a standard criterion to evaluate the performance of any algorithm. A point $(\bar{x},\bar{y})$ is called as the Nash equilibrium, if for all $(x,y) \in \mathcal{X}\times \mathcal{Y}$, it satisfies the condition $f(\bar{x},y) \leq f(\bar{x},\bar{y}) \leq f(x,\bar{y})$. That is, for all $y \in \mathcal{Y}$, $\bar{x}$ minimizes $f(\cdot,y)$ and for all $x\in\mathcal{X}$, $\bar{y}$ maximized $f(x,\cdot)$. Several algorithms exists for efficiently obtaining a point $\epsilon$-close to the Nash equilibrium in the case of offline convex-concave saddle point optimization problem; see, for example~\cite{nemirovski2004prox, Rakhlin2013OptimizationLA}. In the online setting, when we are dealing with static regret, we consider regret notions based on a smoothed version of the above definition of Nash equilibrium; see Definition~\ref{def:ssp} for the exact formulation. This definition of the static regret is motivated by similar notion of smoothed regret for online nonconvex optimization in~\cite{hazan2017efficient} and is also considered in~\cite{2018arXiv180608301R} for the online saddle-point optimization problem. For the case of dynamic regret, we propose natural notions of cumulative regret between the iterates $(x_t,y_t)$ and the points $(x^*_t,y^*_t)$, either in terms of iterates or in terms of function-value at the iterates, under the assumption that function $f_t$ satisfy certain bounded variation conditions; see Definitions~\ref{def:dsppath} and~\ref{def:dsp2} for the exact formulation. Our proposal for the dynamic regret is motivated by similar notions of regret in the online convex and nonconvex optimization setting~\cite{bousquet2002tracking, hall2015online, Besbes_2015,besbes2014stochastic,yang2016tracking,keskin2016chasing, gao2018online, chen2019nonstationary, roy2019multi}. For the proposed notions of static and dynamic regret, we propose and analyze online and bandit variants of extra-gradient method when the sets are unconstrained, i.e., $\mathcal{X}:= \mathbb{R}^{d_X}$ and $\mathcal{Y}:=\mathbb{R}^{d_Y}$. Next, we propose and analyze online and bandit variants of Frank-Wolfe algorithm (designed for saddle-point problems), when the sets $\mathcal{X}\subset \mathbb{R}^{d_X}$, $\mathcal{Y}\subset\mathbb{R}^{d_Y}$ are compact and convex. 

\subsection{Motivating Examples}\label{sec:motiveg}
We now provide our main motivating example (online two-player zero-sum stochastic games) for the type of sequential decision making problems that could be formulated in the form in~\eqref{eq:nssaddle}. Before we proceed, we also highlight that there are several other examples that fall in the framework we consider. We briefly mention them without going into the details. Nonstatioanry saddle point problems also arise when we consider online versions of generative adversarial networks~\cite{grnarova2018an,ge2018fictitious}. Furthermore, by the variational formulation of $l_1$ norm, one could formulate robust versions (where robustness is enforced by considering $l_1$ loss) of online nonparametric prediction~\cite{baby2019online,rakhlin2015online,gaillard2017online} could also be cast in the nonstationary saddle-point formulation we consider. Yet another problem which could be cast in the saddle-point framework is that of maximizing Area under Receiver Operating Characterizing Curves (AUC) in the online setting~\cite{ying2016stochastic}. \\

\noindent\textbf{Online Two-player Zero-sum Stochastic Games:}
One of the main motivating applications for the nonstationary stochastic saddle-point optimization we consider is the problem of two-player zero-sum stochastic games~\cite{mertens1981stochastic, filar2012competitive}. In this setting, we consider two agents, that are characterized by the following Markov Decision Process (MDP) $M$ parametrized by the tuple $(\mathcal{S}, \mathcal{A}_1, \mathcal{A}_2, \mathcal{P}, c)$. Here, $\mathcal{S} \subset \mathbb{R}^b$ is the state space of the MDP. The sets $\mathcal{A}_1 \subset \mathbb{R}^{p_1}$ and $\mathcal{A}_2 \subset \mathbb{R}^{p_2}$ denotes the action space of agent 1 and agent 2 respectively. Furthermore, $\mathcal{P}: \mathcal{S} \times  \mathcal{S} \times \mathcal{A}_1 \times \mathcal{A}_2  \to [0,1]$ denotes be the transition probability kernel and $c(s,a^{(1)},a^{(2)}):\mathcal{S} \times \mathcal{A}_1 \times \mathcal{A}_2 \to \mathbb{R}$ denotes the cost-reward functions corresponding to the actions of agent 1 and agent 2. In the zero-sum setting, the cost-reward  function is setup so that goals of agent 1 and agent 2 are contradictory in nature. For example, the goal of agent 1 is to minimize the cost over time, while the goal of agent 2 is to maximize the cost. This is done by the two agents working with the MDP $M$, at a given time step $t$, by choosing actions $a^{(1)}_t$ and $a^{(2)}_t$ based on data $\{s_i, a^{(1)}_i, a^{(2)}_i, c(s_i,a^{(1)}_i, a^{(2)}_i)\}_{i=1}^{t-1}$ and $s_t$. Based on the actions chosen, the process moves to state $s_{t+1}$ with probability $\mathcal{P}(s_{t+1}|a^{(1)}_t, a^{(2)}_t,s_t)$. To formulate the problem precisely, we introduce the so-called policy function, $\pi_{\theta_1}(a|s) \equiv \pi_{\theta_1}(a,s): \mathcal{A}_1 \times \mathcal{S} \to [0,1]$, which denotes the probability of agent 1 taking action $a$ in state $s$; the policy function for agent 2 is defined similarly is denoted as $\pi_{\theta_2}(a|s)$. Here, $\theta_1 \in \mathbb{R}^{d_1}$ and $\theta_2 \in \mathbb{R}^{d_2}$ are the parameter vectors of the policy function $\pi_{\theta_1}$ and $\pi_{\theta_2}$ respectively. Then, the precise formulation of the problem describing the goal of the agent is given by the following offline optimization problem:
\begin{align*}
(\theta^*_1, \theta^*_2) =  \argmin_{\theta_1 \in \Theta_1}~\argmax_{\theta_2 \in \Theta_2} \left\{ J(\theta_1,\theta_2) = \E_s \left[ V_{\theta_1,\theta_2}(s) \right]  = \E_s \left[\E\left(\sum_{i=1}^t c(s_i,a^{(1)}_i,a^{(2)}_i)\bigg|s_1=s  \right) \right]\right\},
\end{align*}
where $a^{(1)}_i \sim \pi_{\theta_1}(\cdot|s_i)$, $a^{(2)}_i \sim \pi_{\theta_2}(\cdot|s_i)$ and $s_{i+1}\sim \mathcal{P}(\cdot|s_i,a^{(1)}_i, a^{(2)}_i)$, for all $1\leq i < t$ and $\E_s$ represents the expectation with respect to the (fixed) initial distribution of the states. The quantity $V_{\theta_1,\theta_2}(s)$ is called the value function and is indexed by $(\theta_1,\theta_2)$ to represent the fact that it depends on the policy function $\pi_{\theta_1}$ and $\pi_{\theta_2}$. Naturally, the above problem is an offline saddle-point problem.

In the online nonstationary version of the two-player zero-sum stochastic game~\cite{wei2017online}, there are two significant changes to the above setup, which are motivated by similar changes in single-player MDP~\cite{neu2010online, arora2012online, guan2014online, dick2014online}. First, the cost function $c$ is assumed to change with time and is hence indexed by $c_t$. Next, the interaction protocol of the agent is changed so that in time $t$, receives $s_t$ and selects action $a^{(1)}_t,a^{(2)}_t$ based on which it receives the cost $c_t(s_t, a^{(1)}_t,a^{(2)}_t)$. The probability kernel $\mathcal{P}$ is typically assumed to be known in Online MDP problems~\cite{neu2010online, dick2014online}. The goal in online nonstationary MDP is to come up with a sequence of policies $\pi_{\theta^*_{1,t}}, \pi_{\theta^*_{2,t}}$ to minimize an appropriately defined notion of static or dynamic (nonstationary) regret. This falls under the category of sequential decision making problems as described in~\eqref{eq:nssaddle}. When the policies are chosen based on the logistic regression model, the problem becomes a sequential strongly-convex and strongly-concave saddle-point optimization problem and our results in Section~\ref{sec:egmethod} and~\ref{sec:dynamicsection} could potentially be applied to obtain the corresponding regret bounds for online two-player zero-sum stochastic games.

\subsection{Related Work}
\textbf{Offline Saddle-Point Optimization:} Offline saddle-point optimization problems have a long history in the mathematical programming and operations research community. The celebrated extragradient method was proposed in~\cite{korpelevich1976extragradient} and consequently analyzed by~\cite{tseng1995linear,flaam1996equilibrium, facchinei2007finite} for the case of bilinear objectives and strongly-convex and strongly-concave objectives. Generalizing the extragradient method,~\cite{nemirovski2004prox} proposed and analyzed the mirror-prox method for the smooth convex and concave objectives, which was also later analyzed by~\cite{monteiro2010complexity}. A sub-gradient based algorithms was proposed and analyzed in~\cite{nedic2009subgradient} to handle non-smooth objectives. A unified view of extragradient and proximal point method was provided in~\cite{mokhtari2019unified} and a stochastic version of offline saddle-point problems was considered in~\cite{palaniappan2016stochastic}. Frank-Wolfe algorithm for saddle-point optimization was analyzed in~\cite{gidel2017frank}, where it was noted that the first use of Frank-Wolfe algorithm for saddle-point optimization was in~\cite{hammond1984solving}. Recently, there has been an ever-growing interest in analyzing the case of nonconvex-nonconcave objectives, motivated by its applications to training generative adversarial networks. Several works, for example,~\cite{daskalakis2017training, rafique2018non,nouiehed2019solving, sanjabi2018convergence,  flokas2019poincar, lin2019gradient, jin2019minmax, thekumparampil2019efficient}, proposed and analyzed variants of gradient descent ascent for nonconvex-concave objectives and nonconvex-nonconcave objectives.

In the learning theory community, an alternative approach for offline saddle-point optimization problems has been considered. This approach involves using an online convex optimization algorithm for performing offline saddle-point optimization; see, for example~\cite{daskalakis2011near, Syrgkanis2015FastCO,Rakhlin2013OptimizationLA, CesaBianchi2006PredictionLA,Abernethy2017OnFA, Bailey2018MultiplicativeWU} for more details on this approach. In particular,~\cite{Abernethy2017OnFA} established connections between online Frank-Wolfe algorithms and offline saddle-point problems. The developed approaches in the learning theory community compares favorably to the optimal algorithm developed in the mathematical programming and operations research community (for example,~\cite{nemirovski2004prox}).

\vgap
\noindent \textbf{Online Saddle-Point Optimization:} The literature on online saddle-point optimization is extremely limited. Considering the case of deterministic bilinear saddle-point problems, i.e., the case when the function $f_t(x,y):= x^\top A_t y$,~\cite{Cardoso2019CompetingAE} proposed and analyzed algorithms for competing against a notion of static regret. Such bilinear problems arise in online bandit learning problems with knapsack constraint~\cite{Immorlica2018AdversarialBW}. Furthermore~\cite{2018arXiv180608301R} considered online saddle-point problems in strongly-convex and strongly-concave setting and provided regret bounds for a similar notion of static regret as in~\cite{Cardoso2019CompetingAE}. We emphasize that both~\cite{Immorlica2018AdversarialBW} and~\cite{2018arXiv180608301R} only considered the static setting and assumed access to an oracle that computes exact maximization and minimization of convex and concave functions respectively and did not analyze iterative algorithms, as we do in this work.  

\subsection{Our Contributions}
In this work, we consider nonstationary version of stochastic saddle-point optimization problems, in the online and bandit setting and make the following contributions:
\begin{enumerate}
\item We propose natural notions of static regret (Definitions~\ref{def:ssp}) and dynamic regret (Definitions~\ref{def:dsppath} and~\ref{def:dsp2}) that are suited for nonstationary saddle-point optimization problems.  
\item We analyze online and bandit versions of extragradient method for the unconstrained setting and provide bounds for both the static and dynamic regret in Theorem~\ref{th:rsspcor} and~\ref{th:dynamicexgrad} respectively.
\item Next, for the constrained setting, we analyze online and bandit versions of saddle-point Frank-Wolfe method and provide bounds for both static and dynamic regret in Theorem~\ref{th:static_ol_firstOrder} and~\ref{thm:fwdynamicfuncreg} respectively.
\item In the process of establishing the above mentioned result, we also analyze offline zeroth-order saddle-point Frank-Wolfe method and provide results for obtaining $\epsilon$-Nash equilibrium solution in Theorem~\ref{theo:offl_nonadapt}. 
\item Finally, we also consider online and bandit version of gradient descent ascent algorithm for nonstationary saddle-point optimization problem and establish regret bounds for a weaker notion of dynamic regret (Definition~\ref{def:dsp}) in Theorem~\ref{th:absoutsum}. This demonstrates the drawback of gradient descent ascent algorithm for nonstationary saddle-point optimization problems. 
\end{enumerate}
To the best of our knowledge, our results provide the first static and dynamic regret bounds for nonstationary saddle-point optimization problems of the form in~\eqref{eq:nssaddle}.
\subsection{Preliminaries}
\noindent\textbf{Notations:} For a vector $u \in \mathbb{R}^d$, $\|u\|$ always denotes the standard $\ell_2$ norm, unless specified otherwise. For a function $f(x,y)$, we denote by $\nabla_x f(x,y)$ and $\nabla_y f(x,y)$ the partial derivative of $f(x,y)$ with respect to $x$ and $y$ respectively. Throughout the paper, we use $x[y]$ to denote a fact that holds for both variables $x$ and $y$. For example, $\|\nabla_{x[y]} f(x,y)\| \leq B_{x[y]}$ for some $ B_{x[y]} > 0$ means $\|\nabla_{x} f(x,y)\| \leq B_{x}$ and $\|\nabla_{y} f(x,y)\| \leq B_{y}$ for some constants $B_x,B_y > 0$. We will denote the filtration generated up to the $t^{th}$ iteration of any Algorithm in this paper by $\calF_t$. Next we provide the precise assumptions we make on the functions $F_t$ and the different notions of regret we consider. We first start with several regularity assumptions on the function $F_t$. 
\begin{assumption}[\bfseries Strongly-Convex and Strongly-Concave Function]  \label{as:strongconcon}
The objective functions $F_t\left(x,y,\xi\right)$ are continuously differentiable in $x$, and $y$. Moreover, the functions $F_t\left(x,y,\xi\right)$ are $\mu_X$-strongly convex in $x$, and $\mu_Y$-strongly concave in $y$. We also define $\mu:=\min \left(\mu_X,\mu_Y\right)$. 
\end{assumption}
\begin{assumption}[Lipschitz Function] \label{as:lip}
The functions $F_t$ are $L_X$-Lipschitz w.r.t $x$, and $L_Y$-Lipschitz w.r.t $y$, i.e., almost surely, we have $|F_t\left(x_1,y,\xi \right)- F_t\left(x_2,y,\xi \right)| \leq L_X\norm{x_1-x_2}$, and $|F_t\left(x,y_1,\xi \right)- F_t\left(x,y_2,\xi \right)| \leq L_Y\norm{y_1-y_2}$. We also define $L:=\max\left(L_X,L_Y\right)$.
\end{assumption}
\begin{assumption}[\bfseries Lipschitz Gradient]  \label{as:lipgrad}
The functions $F_t$ have Lipschitz continuous gradient w.r.t $x$, and $y$, i.e., almost surely, we have $\norm{\nabla_x F_t\left( x_1,y,\xi\right) -\nabla_x F_t\left( x_2,y,\xi\right)}\leq L_{GX}\|x_1-x_2\|$, for all $y$ and , $\norm{\nabla_y F_t\left( x,y_1,\xi\right) -\nabla_y F_t\left( x,y_2,\xi\right)}\leq L_{GY}\|y_1-y_2\|$ for all $x$. Similar to before, we define $L_G:=\max \left(L_{GX},L_{GY}\right)$.
\end{assumption}
\begin{assumption}[\bfseries Lipschitz Partial Gradient]  \label{as:crosslipgrad}
The functions $F_t$ have Lipschitz continuous partial gradient w.r.t $x$, and $y$, i.e., almost surely, we have $\norm{\nabla_x F_t\left( x,y_1,\xi\right) -\nabla_x F_t\left( x,y_2,\xi\right)}\leq L_{XY}\|y_1-y_2\|$, for all $x$ and , $\norm{\nabla_y F_t\left( x_1,y,\xi\right) -\nabla_y F_t\left( x_2,y,\xi\right)}\leq L_{YX}\|x_1-x_2\|$, for all $y$. 
\end{assumption}
\begin{assumption}[\bfseries Bounded Third-order Derivative]  \label{as:liphess}
The functions $F_t$ have bounded third order partial derivative, i.e., almost surely, we have
\begin{align*}
\max\left(\|\nabla^3_{xxx}F_t\|,\|\nabla^3_{xxy}F_t\|,\|\nabla^3_{xyx}F_t\|,\|\nabla^3_{yxx}F_t\|,\|\nabla^3_{yxy}F_t\|,\|\nabla^3_{xyy}F_t\|,\|\nabla^3_{yyx}F_t\|,\|\nabla^3_{yyy}F_t\|\right)\leq L_H,
\end{align*} 
where $\| \cdot \|$ represents the tensor operator norm; see, for example,~\cite{waterhouse1990absolute} for the definition.
\end{assumption}
We now state the assumptions on the oracles used in this work. The stochastic zeroth-order oracle, where we only observe noisy unbiased function evaluations as feedback, is used in the bandit setting. The stochastic first-order oracle, where one could also obtain noisy unbiased estimators of the gradients, is used in the online setting. 
\begin{assumption}[Stochastic Zeroth-order Oracle]\label{as:stoch}
For any $x \in \mathbb{R}^{d_X}$ and $y \in \mathbb{R}^{d_Y}$, the zeroth-order oracle outputs an estimator $F\left(x,y,\xi\right)$ of $f\left(x,y\right)$ such that such that $\expec{F\left(x,y,\xi\right)}=f\left(x,y\right)$, $\expec{\nabla_{x[y]} F\left(x,y,\xi\right)}=\nabla_{x[y]} f\left(x,y\right)$, and $\expec{\|\nabla_{x[y]} F\left(x,y,\xi\right)-\nabla_{x[y]} f\left(x,y\right)\|^2}\leq \sigma^2$.
\end{assumption}
\begin{assumption}[Stochastic First-order Oracle]\label{as:stoch1} For any $x \in \mathbb{R}^{d_X}$ and $y \in \mathbb{R}^{d_Y}$, the zeroth order oracle outputs an estimator $\nabla_{x[y]} F\left(x,y,\xi\right)$ such that  $\expec{\nabla_{x[y]} F\left(x,y,\xi\right)}=\nabla_{x[y]} f\left(x,y\right)$, with variance bounded as $\expec{\|\nabla_{x[y]} F\left(x,y,\xi\right)-\nabla_{x[y]} f\left(x,y\right)\|^2}\leq \sigma^2$.  
\end{assumption}
We now state the definitions of uncertainty sets capturing the allowed degree of nonstationarity of the functions $f_t$. To do so, first recall the definition of $(x_t^*,y_t^*)$ from~\eqref{eq:nssaddle}. 
\begin{definition}[Optimal Value Variation] \label{def:us3}
For a given $V_T\geq0$, the uncertainty set of functions $\mathcal{M}_T$ is defined as,
\begin{align*}
\mathcal{M}_T(\{f_t\}_{t=1}^T)\stackrel{\tiny\mbox{def}}{=}\left\lbrace \{f_t\}_{t=1}^T:  \sum_{t=1}^{T}\left(\|x_{t+1}^*-x_{t}^*\|^2+\|y_{t+1}^*-y_{t}^*\|^2\right) \leq V_T \right\rbrace. \numberthis \label{eq:us3}
\end{align*}
\end{definition}
\begin{definition}[Optimal Value Variation] \label{def:us3}
For a given $W_T\geq0$, the uncertainty set of functions $\mathcal{D}_T$ is defined as,
\begin{align*}
\mathcal{D}_T(\{f_t\}_{t=1}^T)\stackrel{\tiny\mbox{def}}{=}\left\lbrace \{f_t\}_{t=1}^T:  \sum_{t=1}^{T}\norm{f_t-f_{t+1}}\leq W_T \right\rbrace. \numberthis \label{eq:us4}
\end{align*}
where $\norm{f_t-f_{t+1}}:=\sup_{x,y\in\mathcal{X},\calY}\abs{f_t\left( x,y\right) -f_{t+1}\left( x,y\right)}$.
\end{definition}
The definitions of the above two sets $\mathcal{M}_T$ and $\mathcal{D}_T$ in essence capture the degree of nonstationarity allowed in our problem and are directly motivated by similar assumptions made in the literature on \texttt{argmin}-type online optimization problems~\cite{Besbes_2015, yang2016tracking, gao2018online, roy2019multi}.
As discussed in [BGZ15], the above definitions capture several types of nonstationarity occurring in practice, including continuous changes and discrete shocks. Furthermore, we emphasize that the degree of nonstationarity or uncertainty (captured by $V_T$ or $W_T$) is allowed to change with the horizon $T$.

\section{Algorithms for Nonstationary Saddle-Point Optimization}\label{sec:mainalgos}
We now discuss the extragradient and the Frank-Wolfe algorithms we use for obtaining regret bounds for the nonstationary saddle-point optimization problem in~\eqref{eq:nssaddle}, in Sections~\ref{sec:eggmethod} and~\ref{sec:fwmethod} respectively. We also remark that a natural algorithm for offline saddle-point optimization is the gradient descent ascent algorithm, i.e., alternate between a descent step for the minimization part and an ascent step for the maximization part. It is worth exploring the performance of this algorithm for nonstationary saddle-point optimization problems. In section~\ref{sec:gdadynamic}, we discuss this aspect in detail and highlight the limitations of this approach. Before we proceed, we first discuss the gradient estimators used in the first and zeroth-order setting.

\subsection{Gradient Estimator}
Both extragradient method and Frank-Wolfe method (discussed later in in Algorithm~\ref{alg:EG} and~\ref{alg:FWdyn} respectively) are gradient-based algorithms. When used to solve the nonstationary saddle-point optimization problem in~\eqref{eq:nssaddle}, we require a mini-batch gradient estimators as described in Algorithm~\ref{alg:gradest}. Specifically, the mini-batch gradient estimators used changes depending on if we are bounding static and dynamic notions of regrets, and depending on the availability of stochastic first and zeroth-order oracle information. The interpretation of the gradient estimators in the first-order setting is straightforward as we just query the stochastic first-order oracle in Assumption~\ref{as:stoch1} to obtain noisy but unbiased gradients. The main remark we make in this setting is about the static regret case, where we consider the gradient of the smoothed function as in~\eqref{eq:smoothedfunction}. In the zeroth-order setting, we only assume availability of the noisy function evaluations (as in Assumption~\ref{as:stoch}). Hence, we use the Gaussian Stein's identity based random gradient estimator, a standard gradient estimator in the zeroth-order optimization literature~\cite{duchi2015optimal, nesterov2017random, 2018arXiv180906474B}. 

We now briefly recap the main idea behind this technique for estimating the gradient of a function from noisy evaluations. Let $ u \sim N(0,I_d)$ be a standard Gaussian random vector. Given a function $f(x):\mathbb{R}^d \to \mathbb{R}$, for some $\nu \in (0,\infty)$ consider the smoothed function $f_\nu(\pp) = \E_u \left[f(\pp+ \nu u) \right]$. Nesterov~\cite{nesterov2017random} has shown that $\nabla f_\nu(\pp) =$
\begin{align*}
 \E_u \left[\frac{f(\pp+\nu u)}{\nu}~u\right] = \E_u \left[ \frac{f(\pp+\nu u) - f(\pp)}{\nu}~u\right]  = \frac{1}{(2\pi)^{d/2}} \int \frac{f(\pp+\nu u) - f(\pp)}{\nu}~u ~e^{-\frac{\|u\|^2}{2}}~du.
\end{align*}
In~\cite{2018arXiv180906474B}, this relation was noted to be just an instantiation of Stein's identity popular in statistics literature.  The above relation implies that we can estimate the gradient of $f_\nu$ by only using evaluations of $f$. In particular, one can define stochastic gradient of $f_\nu(\pp)$ as
$ G(x) = \nu^{-1}(F(\pp+ \nu u, \dd) - F(\pp, \dd))u$,
which is an unbiased estimator of $\nabla f_\nu(\pp)$, i.e.,
$\E_{u,\dd}[G(x)]=\nabla f_\nu(\pp)$. Furthermore,~\cite{nesterov2017random} showed that the gradient $\nabla f_\nu(\pp)$ is not too far from the required gradient $\nabla f(\pp)$. A related idea of using uniform random variables as opposed to Gaussian random variable, have also been considered since the work of~\cite{flaxman2005online} for bandit optimization problems. In the context of saddle-point optimization problems that we consider, we leverage the above approach and construct the partial gradient estimators in Algorithm~\ref{alg:gradest}.
\begin{algorithm}[t!]
	\caption{\texttt{gradest}: Gradient Estimator for Nonstationary Saddle-Point Optimization} \label{alg:gradest}
	\textbf{Input:} $x_t \in \mathbb{R}^{d_X}, y_t \in \mathbb{R}^{d_Y}$, $\nu_{X[Y]}>0$, $m_t^{X[Y]}>0$ 
	\begin{algorithmic}[1]
		\State \textbf{if} (Static Regret)
		\State \quad \textbf{if} (Zeroth-order setting)
		\begin{subequations} \label{eq:zerograddefstatic}
			\begin{align} 
			\begin{split}
			\bar{G}_t^x\left(x_t,y_t\right)=\frac{1}{tm_t^X}\sum_{i=1}^{t}\sum_{j=1}^{m_t^X}\frac{F_i\left(x_t+\nu_X u_{i,j,x},y_t,\xi_{i,j,x}\right)-F_i\left(x_t,y_t,\xi_{i,j,x}\right)}{\nu_X}u_{i,j,x}
			\end{split}\\
			\begin{split}
			\bar{G}_t^y\left(x_t,y_t\right)=\frac{1}{tm_t^Y}\sum_{i=1}^{t}\sum_{j=1}^{m_t^Y}\frac{F_i\left(x_t,y_t+\nu_Y u_{i,j,y},\xi_{i,j,y}\right)-F_i\left(x_t,y_t,\xi_{i,j,y}\right)}{\nu_Y}u_{i,j,y}
			\end{split}
			\end{align}
		\end{subequations} 
		\State \quad \textbf{if} (First-order setting)  
		\begin{align*} 
		\bar{G}_t^{x[y]}\left(x_t,y_t\right)=\frac{1}{tm_t}\sum_{i=1}^{t}\sum_{j=1}^{m_t}\nabla_{x[y]}F_i\left(x_t,y_t,\xi_{i,j}\right)
		\end{align*}
		\State \textbf{else}
		\State \quad \textbf{if} (Zeroth-order setting)
		\begin{subequations} \label{eq:zerograddef}
			\begin{align} 
			\begin{split}
			\bar{G}_t^x\left(x_t,y_t\right)=\frac{1}{m_t^X}\sum_{j=1}^{m_t^X}\frac{F_t\left(x_t+\nu_X u_{j,x},y_t,\xi_{j,x}\right)-F_t\left(x_t,y_t,\xi_{j,x}\right)}{\nu_X}u_{j,x}
			\end{split}\\
			\begin{split}
			\bar{G}_t^y\left(x_t,y_t\right)=\frac{1}{m_t^Y}\sum_{j=1}^{m_t^Y}\frac{F_t\left(x_t,y_t+\nu_Y u_{j,y},\xi_{j,y}\right)-F_t\left(x_t,y_t,\xi_{j,y}\right)}{\nu_Y}u_{j,y}
			\end{split}
			\end{align}
		\end{subequations} 
		\State \quad \textbf{if} (First-order setting)
		\begin{align*} 
		\bar{G}_t^{x[y]}\left(x_t,y_t\right)=\frac{1}{m_t}\sum_{j=1}^{m_t}\nabla_{x[y]}F_t\left(x_t,y_t,\xi_{j}\right)
		\end{align*}
		\State \textbf{end if}
		\State \textbf{Output:} $\left[\bar{G}_t^x\left(x_t,y_t\right);\bar{G}_t^x\left(x_t,y_t\right)\right]$
	\end{algorithmic}	
\end{algorithm}
Similar to the first-order setting, if we are dealing with static regret, we calculate the zeroth-order estimator of smoothed function from~\eqref{eq:smoothedfunction}, as provided in~\eqref{eq:zerograddefstatic}. We end this section with the following results, which are essentially a zeroth-order gradient estimation error result from~\cite{2018arXiv180906474B}, adapted to the saddle-point problem. 
\begin{lemma}[\cite{2018arXiv180906474B}]\label{lm:zerograd}
	Let, Assumptions~\ref{as:lip}, and \ref{as:lipgrad} are true for function $f$, and $\Delta_t^{x[y]}\left(x_t,y_t\right)=\bar{G}_t^{x[y]}\left(x_t,y_t\right)-\nabla_{x[y]}f\left(x_t,y_t\right)$ where $\bar{G}_t^{x[y]}\left(x_t,y_t\right)$ are as defined in \eqref{eq:zerograddef}. Then,
	\begin{subequations}
	\begin{align}
	\begin{split} \label{eq:grad_est_error}
	\expec{\left\lVert\Delta_t^{x[y]}\right\rVert^2}\leq \frac{2\left(d_{X[Y]}+5\right)\left(L^2+\sigma^2\right)}{m_t^{X[Y]}}+\frac{3\nu_{X[Y]}^2}{2}L_G^2\left(d_{X[Y]}+3\right)^3
	\end{split}\\
	\begin{split} \label{eq:gradnormbound}
	 \expec{\left\lVert\bar{G}_t^{x[y]}\right\rVert^2}\leq \frac{\nu_{X[Y]}^2L^2}{2m_t^{X[Y]}}\left(d_{X[Y]}+6\right)^3+\frac{2}{m_t^{X[Y]}}\left(L^2+\sigma^2\right)\left(d_{X[Y]}+4\right).
	\end{split}
	\end{align}
	\end{subequations}
\end{lemma}
We remark that we will eventually use the result in Lemma~\ref{lm:zerograd} to handle the estimation error of the zeroth-order gradient in~\eqref{eq:zerograddefstatic}.

\subsection{Extragradient Algorithm for Nonstationary Saddle-Point Optimization}\label{sec:eggmethod}
It is well-documented that the simple gradient descent algorithm suffers form several convergence issues even in the case of offline saddle point optimization problems with strongly-convex and strongly-concave objectives~\cite{mokhtari2019unified}. It is also known that in comparison to gradient descent ascent, extragradient algorithm~\cite{korpelevich1976extragradient} has superior performance guarantees for offline saddle-point optimization problems~\cite{tseng1995linear, mokhtari2019unified}. We now discuss a version of Extragradient algorithm, displayed in Algorithm~\ref{alg:EG}, suitable for nonstationary saddle-point optimization problems as in~\eqref{eq:nssaddle}, in the unconstrained setting.
\begin{algorithm}[t]
	\caption{Extragradient Method (\texttt{EG}) for Nonstationary Saddle-Point Optimization} \label{alg:EG}
	\textbf{Input:} $\eta_t$, $x_0\in \mathbb{R}^{d_X}$, $y_0 \in \mathbb{R}^{d_Y}$, $\nu_{X[Y]}>0$
	\begin{algorithmic}[1]
		\State \textbf{for} $t=0,1,\cdots,T-1$ \textbf{do}
		\State \textbf{Set} $\left[\bar{G}_t^x\left(x_t,y_t\right);\bar{G}_t^y\left(x_t,y_t\right)\right]=\texttt{gradest}\left(x_t,y_t,\nu_{X[Y]},m_t^{X[Y]}\right)$ \Comment{ Algorithm~\ref{alg:gradest}}
		\State\textbf{Set} $x_{t+\half}=x_t-\eta_t\bar{G}_t^x\left(x_t,y_t\right)$ \qquad $y_{t+\half}=y_t+\eta_t\bar{G}_t^y\left(x_t,y_t\right)$
		\State \textbf{Set} $\left[\bar{G}_t^x\left(x_{t+\frac{1}{2}},y_{t+\frac{1}{2}}\right);\bar{G}_t^y\left(x_{t+\frac{1}{2}},y_{t+\frac{1}{2}}\right)\right]=\texttt{gradest}\left(x_{t+\frac{1}{2}},y_{t+\frac{1}{2}},\nu_{X[Y]}, m_t^{X[Y]}\right)$ \Comment{Algo~\ref{alg:gradest}}
		\State\textbf{Update} $x_{t+1}=x_t-\eta_t\bar{G}_t^x\left(x_{t+\half},y_{t+\half}\right)$ \qquad $y_{t+1}=y_t+\eta_t\bar{G}_t^y\left(x_{t+\half},y_{t+\half}\right)$
		\State \textbf{end for}
	\end{algorithmic}	
\end{algorithm}
The main idea underlying the extragradient algorithm is the use of the additional gradient descent step, as in step 3 of Algorithm~\ref{alg:EG}. As shown in~\cite{mokhtari2019unified}, this step, when used in the offline setting approximates the computationally prohibitive proximal point method, to the required amount of accuracy, at the same time being practically efficient. In order to leverage this property of extragradient method for nonstationary stochastic saddle-point optimization problems, the main modification which we require is the use of mini-batch gradient estimators (as described in Algorithm~\ref{alg:gradest}) in step 2 and step 4 of Algorithm~\ref{alg:EG}. We show in Section~\ref{sec:staticegmethod} and~\ref{sec:dynamicegmethod}, by selecting the batch sizes appropriately, one could obtain sub-linear regret bounds in both the static and dynamic setting. 

As mentioned above, one way to think of the extragradient method is as an approximation to the proximal point method. In Algorithm~\ref{alg:pp}, we provide a \emph{fictional} algorithm which allows us to connect the iterates of the proximal point method and the extragradient method in the context of nonstationary saddle-point optimization problems of the form~\eqref{eq:nssaddle}. We emphasize that Algorithm~\ref{alg:pp} is only for the purpose of proving our regret bounds in Section~\ref{sec:staticegmethod} and~\ref{sec:dynamicegmethod} later. We now connect the iterate of the proximal point-type method in Algorithm~\ref{alg:pp} and extragradient method in Algorithm~\ref{alg:EG}.

\begin{lemma} \label{lm:stocherrorppeg}
	Let $\left(x_{t+1},y_{t+1}\right)$, $\left(\tx_{t+1},\ty_{t+1}\right)$ and $\left(\hat{x}_{t+1},\hat{y}_{t+1}\right)$ be the updates generated by zeroth-order stochastic extra-gradient method, first-order deterministic extra-gradient method, and proximal point-type method (as in Algorithm~\ref{alg:pp}) respectively from $\left(x_{t},y_{t}\right)$ for a function $f_t$ for which Assumptions~\ref{as:strongconcon}--\ref{as:lipgrad}, and \ref{as:liphess} are true. Then:
	\begin{enumerate}[label=(\alph*)]
		\item In the zeroth-order setting,
		\begin{align}
		\expec{\|x_{t+1}-\hat{x}_{t+1}\|^2|\calF_t}\leq e_{0,t+1,X} \quad \expec{\|y_{t+1}-\hat{y}_{t+1}\|^2|\calF_t}\leq e_{0,t+1,Y},	\label{eq:stocherrorppega}
		\end{align}
		where $e_{0,t+1,{X[Y]}}=4\left(\eta_t^2+L_G^2\eta_t^4\right)\left(\frac{2\left(d_{X[Y]}+5\right)\left(L^2+\sigma^2\right)}{m_t^{X[Y]}}+\frac{3\nu_{X[Y]}^2}{2}L_G^2\left(d_{X[Y]}+3\right)^3\right)+2L_H ^2\eta_t^6$.
		\item In the first-order stochastic setting,
		\begin{align}
		\expec{\|x_{t+1}-\hat{x}_{t+1}\|^2|\calF_t}\leq e_{1,t+1} \quad \expec{\|y_{t+1}-\hat{y}_{t+1}\|^2|\calF_t}\leq e_{1,t+1}, \label{eq:stocherrorppegb}
		\end{align}
		where $e_{1,t+1}=4\left(\eta_t^2+L_G^2\eta_t^4\right)\frac{\sigma^2}{m_t}+2L_H ^2\eta_t^6$.
	\end{enumerate}
\end{lemma}
\begin{proof}[Proof of Lemma~\ref{lm:stocherrorppeg}] First note that we have
 \begin{align*}
	&~\|x_{t+1}-\tilde{x}_{t+1}\|\\
	=&~\eta_t\norm{\bar{G}_t\left(x_{t+\half},y_{t+\half}\right)-\nabla_xf\left(\tilde{x}_{t+\half},\tilde{y}_{t+\half}\right)}\\
	\leq&~\eta_t\norm{\bar{G}_t\left(x_{t+\half},y_{t+\half}\right)-\nabla_xf\left({x}_{t+\half},{y}_{t+\half}\right)}+\eta_t\norm{\nabla_xf\left({x}_{t+\half},{y}_{t+\half}\right)-\nabla_xf\left(\tilde{x}_{t+\half},\tilde{y}_{t+\half}\right)}\\
	\leq &~\eta_t\norm{\bar{G}_t\left(x_{t+\half},y_{t+\half}\right)-\nabla_xf\left({x}_{t+\half},{y}_{t+\half}\right)}+L_G\eta_t^2\norm{\bar{G}_t\left(x_t,y_t\right)-\nabla_xf\left(x_t,y_t\right)}.
	\end{align*}
	Hence, we obtain 
	\begin{align*}
	\|x_{t+1}-\tilde{x}_{t+1}\|^2&\leq 2\eta_t^2\norm{\bar{G}_t\left(x_{t+\half},y_{t+\half}\right)-\nabla_xf\left({x}_{t+\half},{y}_{t+\half}\right)}^2\\
	&~~~+2L_G^2\eta_t^4\norm{\bar{G}_t\left(x_t,y_t\right)-\nabla_xf\left(x_t,y_t\right)}^2.
	\end{align*}
	Now we invoke the following result from \cite{mokhtari2019unified}.
	\begin{proposition}[Proposition 2 in \cite{mokhtari2019unified}] \label{prop:mokh}
	Given a point $(x_t, y_t)$, let $\left(\hx_{t+1}, \hy_{t+1}\right)$ be the point we obtain by performing the Proximal Point update on $(x_t, y_t)$, and let $\left(x_{t+1}, y_{t+1}\right)$ be the point we obtain by performing the Extragradeint update on $\left(x_{t},y_{t}\right)$. Then, for a given stepsize $\eta_t > 0$ we have
	\begin{align*}
	    \|x_{t+1}-\hx_{t+1}\|\leq o(\eta_t^2)\quad \|y_{t+1}-\hy_{t+1}\|\leq o(\eta_t^2)
	\end{align*}
	\end{proposition}
	Using Proposition~\ref{prop:mokh}, and Assumption~\ref{as:liphess} we get,
	\begin{align}
	\|\tilde{x}_{t+1}-\hat{x}_{t+1}\|^2\leq & L_H^2\eta_t^6 \label{eq:lmstocherrorppegc}\\
	\expec{\|x_{t+1}-\hat{x}_{t+1}\|^2|\calF_t}\leq& 2\left(\expec{\|x_{t+1}-\tilde{x}_{t+1}\|^2|\calF_t}+\expec{\|\tilde{x}_{t+1}-\hat{x}_{t+1}\|^2|\calF_t}\right). \label{eq:lmstocherrorppegd}
	\end{align}
In order to complete the proof, we consider the zeroth-order and first-order setting separately:
	\begin{enumerate}[label=(\alph*)]
    \item In the zeroth-order stochastic setting, using \eqref{eq:grad_est_error} we have,
	\begin{align}
	\expec{\|x_{t+1}-\tilde{x}_{t+1}\|^2|\calF_t}\leq 2\left(\eta_t^2+L_G^2\eta_t^4\right)\left(\frac{2\left(d_{X[Y]}+5\right)\left(L^2+\sigma^2\right)}{m_t^{X[Y]}}+\frac{3\nu_{X[Y]}^2}{2}L_G^2\left(d_{X[Y]}+3\right)^3\right) \label{eq:lmstocherrorppega}
	\end{align}
	Combining \eqref{eq:lmstocherrorppegc}, \eqref{eq:lmstocherrorppegd}, and \eqref{eq:lmstocherrorppega}, we get \eqref{eq:stocherrorppega}.
	\item In the first-order stochastic setting, using \eqref{eq:grad_est_error}, we get
	\begin{align}
	\expec{\|x_{t+1}-\tilde{x}_{t+1}\|^2|\calF_t}\leq 2\left(\eta_t^2+L_G^2\eta_t^4\right)\frac{\sigma^2}{m_t} \label{eq:lmstocherrorppegb}
	\end{align}
    Combining \eqref{eq:lmstocherrorppegc}, \eqref{eq:lmstocherrorppegd}, and \eqref{eq:lmstocherrorppegb}, we get \eqref{eq:stocherrorppegb}.
\end{enumerate}
\end{proof}
\begin{remark} 
Although Lemma~\ref{lm:stocherrorppeg}, is proved for the sequence of functions $\lbrace f_t\rbrace_{t=1}^{T}$, by a similar proof, the same result could be shown to hold for the sequence of functions $\lbrace \calJ_t\rbrace_{t=1}^{T}$ (as they satisfy Assumptions~\ref{as:strongconcon}--\ref{as:lipgrad}, and \ref{as:liphess}). 
\end{remark}
\begin{algorithm}[t]
	\caption{Proximal Point-Type Method} \label{alg:pp}
	\textbf{Input:} $\eta_t$,$x_0 \in \mathbb{R}^{d_X} , y_0 \in \mathbb{R}^{d_Y}$ 
	\begin{algorithmic}[1]
		\State \textbf{for} $t=0,2,\cdots,T-1$ \textbf{do}
		\State \textbf{if} (Regret=Static)
		\State\textbf{Update} $\hx_{t+1}=x_t-\eta_t\nabla_x \calJ_t\left(\hx_{t+1},\hy_{t+1}\right)$ \qquad $\hy_{t+1}=y_t+\eta_t\nabla_y \calJ_t\left(\hx_{t+1},\hy_{t+1}\right)$ 
		\State \textbf{else}
		\State\textbf{Update} $\hat{x}_{t+1}=x_t-\eta_t\nabla_x f_t\left(\hat{x}_{t+1},\hat{y}_{t+1}\right)$ \qquad $\hat{y}_{t+1}=y_t+\eta_t\nabla_y f_t\left(\hat{x}_{t+1},\hat{y}_{t+1}\right)$
		\State \textbf{end if}
		\State \textbf{end for}
	\end{algorithmic}	
\end{algorithm}

\subsection{Frank-Wolfe Algorithm for Nonstationary Saddle-Point Optimization}\label{sec:fwmethod}
We also analyze Frank-Wolfe algorithms for nonstationary saddle-point optimization when the problem in~\eqref{eq:nssaddle} is constrained and the sets $\mathcal{X}$ and $\mathcal{Y}$ are compact and convex subsets of $\mathbb{R}^{d_X}$ and $\mathbb{R}^{d_Y}$ respectively. Although proposed as early as 1950s by~\cite{fwold}, it has regained interest in the machine learning and optimization communities due to its wide applicability; we refer the reader to the recent survey~\cite{fwsurvey} for more details. For the case of deterministic offline saddle-point optimization problems,~\cite{gidel2017frank} provided an analysis of Frank-Wolfe algorithm in terms of convergence to the saddle-point. In Algorithm~\ref{alg:FWdyn}, we provide a version of Frank-Wolfe algorithm suitable for nonstationary stochastic saddle-point optimization. Similar to the extragradient method in Algorithm~\ref{alg:EG}, we require mini-batch gradient estimators to handle the stochastic settings we consider. In Section~\ref{sec:statfuncvalfw} and~\ref{sec:dynfuncvalfw}, we show that by appropriately selecting the batch size, one could obtain sub-linear regret bounds in both static and dynamic settings, under appropriate assumptions. 

\begin{algorithm}[t]
	\caption{Frank Wolfe Method (\texttt{FW}) for Nonstationary Saddle-Point Optimization} \label{alg:FWdyn}
	\textbf{Input:} $\gamma_t, x_0 \in \mathbb{R}^{d_X}, y_0 \in \mathbb{R}^{d_Y}$, $\nu_{X[Y]}>0$
	\begin{algorithmic}[1]
		\State \textbf{for} $t=0,2,\cdots,T-1$ \textbf{do}
	\State \textbf{Set} $\left[\bar{G}_t^x\left(x_t,y_t\right);\bar{G}_t^y\left(x_t,y_t\right)\right]=\texttt{gradest}\left(x_t,y_t,\nu_{X[Y]},m_t^{X[Y]}\right)$\Comment{ Algorithm~\ref{alg:gradest}}
		\State\textbf{Define} $r_t=\begin{pmatrix}\bar{G}_t^x\left(x_t,y_t\right)\\-\bar{G}_t^y\left(x_t,y_t\right)\end{pmatrix}$
		\State\textbf{Calculate} $s_t=\argmin_{z\in\calX \times \calY}\langle z,d_t\rangle$
		\State\textbf{Update}
		\begin{align}
		    z_{t+1}=\left(1-\gamma_t\right)z_t+\gamma_t s_t \label{eq:alg2update}
		\end{align}
		\State \textbf{end for}
	\end{algorithmic}	
\end{algorithm}

\subsubsection{Zeroth-Order Offline Saddle Point Optimization}
Recall that~\cite{gidel2017frank} analyzed Frank-Wolfe algorithm for deterministic and offline saddle-point optimization problems.  For our analysis Algorithm~\ref{alg:FWdyn} in Section~\ref{sec:statfuncvalfw} and~\ref{sec:dynfuncvalfw}, we require an understanding of the stochastic zeroth-order version of the Frank-Wolfe algorithm from~\cite{gidel2017frank}. The corresponding stochastic zeroth-order Frank-Wolfe algorithm for offline saddle-point optimization is presented in Algorithm~\ref{alg:zeroordspfw}. In Theorem~\ref{theo:offl_nonadapt}, we provide the oracle complexity of this algorithm, extending the results of Frank-Wolfe method in~\cite{gidel2017frank} for deterministic saddle-point optimization problems and zeroth-order stochastic Frank-Wolfe algorithms in~\cite{2018arXiv180906474B} for regular convex and nonconvex optimization problems. This analysis would be useful later on to analyze the online setting. 

We start with a few notations. Consider the saddle-point optimization problem in~\eqref{eq:nssaddle} in the offline setting (i.e., T=1) with $\mathcal{X}$ and $\mathcal{Y}$ being closed and convex. In this case, we denote the function $f_1$ as just $f$. We assume that the saddle point $(x^*, y^*)$ belongs to the interior of $\calX \times \calY$. We also define the following notion of border distance, standard in the analysis of Frank-Wolfe algorithms~\cite{gidel2017frank}: $\delta_\calX:= \min_{s\in \partial \calX}\norm{x^* - s}$ and $\delta_\calY:= \min_{s\in \partial \calY}\norm{y^* - s}$, where $\partial \calX$ and $\partial \calY$ are the boundaries of convex set $\calX,\calY$, i.e., $\partial \calX = \textsc{closure}(\mathcal{X})\backslash \textsc{interior}(\mathcal{X})$ and similarly for $\calY$. We also assume $\calX,\calY$ are bounded, i.e., $\sup_{x,x'\in \calX}\norm{x-x'}\leq D_\calX$ and $\sup_{y,y'\in \calY}\norm{y-y'}\leq D_\calY$. The zeroth-order Frank-Wolfe algorithm is stated in Algorithm~\ref{algo:offline_SPFW}. The main difference between the offline deterministic Frank-Wolfe algorithm in~\cite{gidel2017frank} and our Algorithm~\ref{algo:offline_SPFW} is that the use of zeroth-order random gradient estimators leads to biased estimates of gradients, which needs to be handled differently in our analysis. Note also that, there are two different choices of step-size parameter $\gamma_k$ in Algorithm~\ref{algo:offline_SPFW}. For the case of varying step-size choice, we output a random iteration which is crucial for obtaining our results. Before we proceed with our main result, we introduce few more definitions that are standard in analysis of Frank-Wolfe algorithms. 

\begin{definition}\label{def:fwgaps}
For any iteration $k \geq 1$, the Frank-Wolfe gap corresponding to the minimization part ($x$) and maximization part ($y$) is defined as
\begin{align*}
    &\wh{g}_k^x = -\inner{\derivx{x_{k},y_k},\wh{s}_{k+1}^x-x_k}\\
    &\wh{g}_k^y = -\inner{ -\derivy{x_{k},y_k}, \wh{s}_{k+1}^y - y_k}\\
    &\wh{g}_k = \wh{g}_k^x + \wh{g}_k^y 
\end{align*}
where $\wh{s}_{k+1}^x = \argmin_u \inner{u, \derivx{x_{k},y_k}} \text{ and } \wh{s}_{k+1}^y = \argmax_u \inner{u, \derivy{x_{k},y_k}}$. Furthermore, the merit function, following~\cite{gidel2017frank}, is defined as $w_k:=w_k^x+w_k^y= f(x_k,y^*) - f(x^*,y_k)$, where $w_k^x = f(x_k,y^*) - f^*$ and $w_k^y = f^* - f(x^*,y_k)$. Here $f^*:= f(x^*,y^*)$. 
\end{definition}
We also require the following parameters which appear throughout the proof, which are functions of the gradient estimator's batch size $m_k:=(m_k^{X}, m_k^{Y})$ and smoothing parameters $\nu:=(\nu_{X},\nu_Y)$. With $C_0$ defined in Theorem~\ref{theo:offl_nonadapt}, we define $C_1 = (L_GD^2_\calX+L_GD^2_\calY)/2$, 
\begin{align*}
    C_{2}(m_k,\nu) &= \frac{1}{4C_1}D^2_\calX \left[\frac{4(d_X+5)(L^2_X + \sigma^2)}{m_k^X} + \frac{3(\nu_X)^2}{2}(L_{GX})^2(d_X+6)^3 \right]\\ &\quad\quad +\frac{1}{4C_1}D^2_\calY \left[\frac{4(d_Y+5)(L^2_Y + \sigma^2)}{m_k^Y} + \frac{3(\nu_Y)^2}{2}(L_{GY})^2(d_Y+6)^3 \right],\\
    C_{3}(m_k,\nu) &= \frac{C_0}{2D_\calX} \sqrt{\frac{4(d_X+5)(L^2_X + \sigma^2)}{m_k^X} + \frac{3(\nu_X)^2}{2}(L_{GX})^2(d_X+6)^3} ]\\ &\quad\quad+\frac{C_0}{2D_\calY} \sqrt{\frac{4(d_Y+5)(L^2_Y + \sigma^2)}{m_k^Y} + \frac{3(\nu_Y)^2}{2}(L_{GY})^2(d_Y+6)^3 },\\
    C_{4}(m_k,\nu) &= \frac{4(d_X+5)(L^2_X + \sigma^2)}{m_k^X} + \frac{3(\nu_X)^2}{2}(L_{GX})^2(d_X+6)^3\\&~~~ + \frac{4(d_Y+5)(L^2_Y + \sigma^2)}{m_k^Y} + \frac{3\nu_Y^2}{2}(L_{GY})^2(d_Y+6)^3.
\end{align*}
When the parameter $m_k$ is constant over the iterations, we denote $C_2(m_k,\nu)$, $C_3(m_k,\nu)$, and $C_4(m_k,\nu)$ as just $C_2$, $C_3$, and $C_4$ respectively. We now state our main result on offline zeroth-order Frank-Wolfe algorithm for saddle-point optimization problem, in Theorem~\ref{theo:offl_nonadapt} below. The proof is provided in Section~\ref{sec:proofofzofw}.

\begin{algorithm}[t] \label{alg:zeroordspfw}
\caption{Zero Order Stochastic Saddle Point Frank-Wolfe Algorithm}
\textbf{Input:} $z_0 \in \calX \times \calY$, smoothing parameter $\nu_{X[Y]}>0$, positive integer sequence $m_k$, iteration limit $N \geq 1$ and probability distribution $P_R(\cdot)$ over $\{1,\ldots ,N\}$. \\
\begin{algorithmic}[1]
\State Let $z_0 = (x_0,y_0) \in \calX \times \calY$
\For{$k=1,2,\ldots N$}
    \State Compute $\bar{G}_k: = \texttt{ZOG}(x_{k-1},y_{k-1},m_k^{X[Y]},\nu_{X[Y]})$ \Comment{ Algorithm~\ref{alg:gradestfw}}
    \State Compute $s_k: = \argmin_{u\in \calX\times\calY} \inner{u,\bar{G}_k}$ and  $g_{k-1}: = \langle-\bar{G}_k, s_k-z_{k-1} \rangle$
    \State Let $\gamma_k=\frac{6}{5+k}$ \text{ (non-adaptive step size )} or $\gamma_k=\min \{\frac{C_0}{4C_1}g_{k-1},1\} $  \text{ (adaptive step size )}
    \State Update $z_{k}: = (1-\gamma_k)z_{k-1} + \gamma_k s_k$
\EndFor
\State \textbf{Output:} Generate $R$ according to $P_R(\cdot)$ and output $z_R$\ \text{ (non-adaptive step size )} , or output $z_N$  \text{ (adaptive step size )}
\end{algorithmic}
\label{algo:offline_SPFW}
\end{algorithm}
\begin{algorithm}[t]
	\caption{Zero Order Gradient Estimate (\texttt{ZOG})}  \label{alg:gradestfw}
	\textbf{Input:} $x_{k-1} \in \mathbb{R}^{d_X}$, $y_{k-1}\in\mathbb{R}^{d_Y}$, $\nu_{X[Y]}>0$, $m_k>0$ 
	\begin{algorithmic}[1]
	    \State Generate $u = [u_{1},\ldots,u_{m_k}]$, where $u_{j,x}\sim N(0,I_{d_X}),u_{j,y}\sim N(0,I_{d_Y})$
		\begin{subequations} 
			\begin{align*} 
			\begin{split}
			\bar{G}_k^x\left(x_{k-1},y_{k-1}\right)=\frac{1}{m^X_k}\sum_{j=1}^{m_k^X}\frac{F\left(x_{k-1}+\nu_X u_{j,x},y_{k-1},\xi_{j,x}\right)-F\left(x_{k-1},y_{k-1},\xi_{j,x}\right)}{\nu_X}u_{j,x}
			\end{split}\\
			\begin{split}
			\bar{G}_k^y\left(x_{k-1},y_{k-1}\right)=\frac{1}{m_k^Y}\sum_{j=1}^{m^Y_k}\frac{F\left(x_{k-1},y_{k-1}+\nu_Y u_{j,y},\xi_{j,y}\right)-F\left(x_{k-1},y_{k-1},\xi_{j,y}\right)}{\nu_Y}u_{j,y}
			\end{split}
			\end{align*}
		\end{subequations} 
		\State \textbf{Output:} $\left[\bar{G}_k^x\left(x_{k-1},y_{k-1}\right);\bar{G}_k^y\left(x_{k-1},y_{k-1}\right)\right]$
	\end{algorithmic}	
\end{algorithm}

\begin{theorem}
\label{theo:offl_nonadapt}
Let $F$ be a function for which Assumptions~\ref{as:strongconcon}, and \ref{as:lipgrad} are true. Let $\cX\times \cY$ be a convex and compact set. Let the saddle point of $F$ belongs to the interior of $\cX\times \cY$ and $\delta_\mu:= \sqrt{ \min \{ \mu_X\delta_\calX, \mu_Y\delta_\calY\} }$. Let 
\begin{align*}	
C_0:= 1- \frac{\sqrt{2}}{\delta_\mu}\max \left\{\frac{D_\calX L_{XY}}{\sqrt{\mu_y}},\frac{D_\calY L_{YX}}{\sqrt{\mu_x}}\right\}, B^{L\sigma}_X:= \max\left\{ \frac{\sqrt{L_X^2 + \sigma^2}}{L_{GX}}, 1\right\}, B^{L\sigma}_Y:= \max\left\{ \frac{\sqrt{L_Y^2 + \sigma^2}}{L_{GY}}, 1\right\}.
\end{align*}
If $C_0 >0$, we have the following to be true:
\begin{enumerate}[label=(\alph*)]
	\item For the case of non-adaptive step-size choice for $\gamma_k$, by choosing
	\begin{align*}
	 m_k^{X[Y]} &= B^{L\sigma}_{X[Y]}(d_{X[Y]}+5)N^2, \qquad \nu_{X[Y]} =\sqrt{\frac{B^{L_G\sigma}_{X[Y]}}{2N^2(d_{X[Y]}+6)^3}}\\
	 \gamma_k &= \frac{6}{5+k}, \qquad P_R(R=k)= \frac{\gamma_k\Gamma_N}{2\gamma_k (1-\Gamma_N)}  \numberthis\label{eq:gammamtnutPRchoice_nonadp}
		\end{align*}
		where $\Gamma_k = \prod_{k=1}^N \left(1-\gamma_k/2\right)$, we get 
		\begin{align*}
		\E[w_T] + \E[\wh{g}_R]  \leq \frac{120 w_0}{(N+3)^3} + \frac{18L(D^2_\calX+D^2_\calY)}{C_0(N+5)} + \frac{11\left(\sqrt{(L_{GX})^2+\sigma^2}+\sqrt{(L_{GY})^2+\sigma^2}\right)}{2NC_0}.
		\end{align*}
		Hence, the total number of calls to the stochastic zeroth-order oracle and linear optimization oracle required to find an $\epsilon$-optimal Nash equilibrium solution of problem are, respectively, bounded by 
		$\order((d_X+d_Y)/\epsilon^3)$ and $\order(1/\epsilon)$.
		\item  For the case of adaptive step-size choice for $\gamma_k$, by choosing
		\begin{align} 
		\label{eq:gammamtnutPRchoice_adp}
		\gamma_k = \max\left\{ \frac{C_0}{4C_1}g_{k-1},1\right\},
		\end{align}
		where $g_{k-1}: = \langle-\bar{G}_k, s_k-z^{k-1}\rangle$, as stated in the Algorithm~\ref{alg:zeroordspfw}, and $m_k^{X[Y]}$, $\nu_{X[Y]}$ as in \eqref{eq:gammamtnutPRchoice_nonadp}, we get,
		\begin{align}
		&\E[w_k] \leq (1-\rho)^T\left[\E[w_0] - \frac{1}{\rho}\left(\max\{C_2,C_3\} + C_4\right) \right] + \frac{1}{\rho}\left(\max\{C_2,C_3\} + C_4\right) \label{eq: geom_FW_form}
		\end{align}
		where $$\rho:= 1- \min\left\{\frac{C^2_0\delta^2_\mu}{8C_1}, \frac{C_0}{2}\right\}.$$ 
		This implies that for any $\epsilon$, if we choose $$m_k^{X[Y]} = \frac{B^{L\sigma}_{X[Y]}(d_{X[Y]}+5)}{\epsilon^{2}}, \quad  \nu_{X[Y]} = \sqrt{\frac{B^{L\sigma}_{X[Y]}}{2\epsilon^{-2}(d_{X[Y]}+6)^3}},$$
		the total number of calls to the zeroth-order stochastic oracle and linear optimization oracle required to find an $\epsilon$-optimal solution, respectively, are bounded by 
		 $\order((d_X+d_Y)/\epsilon^2)$ and $\order(\log( 1/\epsilon))$.
	\end{enumerate}
\end{theorem}
\begin{remark}
The number of calls to the zeroth-order oracles and the linear optimization oracle are much improved with the adaptive step-size choice. In particular we obtain the so-called geometric rate of convergence~\cite{gidel2017frank}, with this choice of step-size, in the zeroth-order setting.
\end{remark}

%% file: staticregret.tex
\section{Static Regret Bounds for Saddle-Point Optimization}\label{sec:egmethod}
Static regret refers to the case when one wants to minimize a notion of cumulative regret, which corresponds to the best possible decisions in the hindsight when we know the function $\{f_t\}_{t=1}^T$, \emph{a priori}. This notion of regret is directly motivated by online convex optimization literature; see for example~\cite{CesaBianchi2006PredictionLA, hazan2016introduction}. Recall the definition of $(x_t^*,y_t^*)$ in~\eqref{eq:nssaddle}, which corresponds to the individual minimizers of the functions $f_t$. We first introduce the following definition of $(u_t^*,v_t^*)$,  which corresponds to the saddle-points of certain smoothed (or averaged) functions. For each $t$, define the smoothed function as
\begin{align}\label{eq:smoothedfunction}
\calJ_t\left(x,y\right):=\frac{1}{t}\sum_{i=1}^{t}\expec{F_i\left(x,y,\xi\right)}.
\end{align}
The above notion will prove to be useful for defining static regret and will also be useful in the proofs later. We also define the notion of saddle-points of the smoothed functions as follows:
\begin{align}\label{eq:smoothsaddle}
\left(u_{t}^*,v_{t}^*\right):=\argmin_{x \in \calX}\argmax_{y \in \calY}\calJ_{t-1}\left(x,y\right).
\end{align}
With this definition, we have the following definition of static regret.
\begin{definition} \label{def:ssp}
Let $F_t$ be a suence functions satisfying Assumption~\ref{as:strongconcon}. Then note that $f_t$ also satisfy Assumption~\ref{as:strongconcon}. For this class of functions, with $(u_{t}^*,v_{t}^*)$ as defined in~\eqref{eq:smoothsaddle}, the Static Saddle-Point (SSP) Regret is defined as
	\begin{align}
	\mathfrak{R}_{SSP}:=\expec{\abs{\sum_{t=1}^{T}f_t\left(x_t,y_t\right)-\sum_{t=1}^{T}f_t\left(u_{T+1}^*,v^*_{T+1}\right)}}.
	\end{align}
\end{definition}
In the context of online saddle-point optimization, the above notions of regret was also considered in~\cite{Cardoso2019CompetingAE} for bi-linear functions and in~\cite{2018arXiv180608301R} for strongly-convex and strongly-concave functions. But both~\cite{2018arXiv180608301R} and~\cite{Cardoso2019CompetingAE} did not consider the case of nonstationary functions with bounded variations and assumed access to exact minimization and maximization oracles.

\subsection{Static Regret Bounds for Extragradient Method}\label{sec:staticegmethod}
We now establish sub-linear bounds on the static regret as in Definition~\ref{def:ssp}, for the extragradient method. Here, we assume the problem~\eqref{eq:nssaddle} is unconstrained, i.e., $\calX= \mathbb{R}^{d_X}$ and $\calY=\mathbb{R}^{d_Y}$.  
\begin{theorem} \label{th:rsspcor}
	Let $\left(x_t,y_t\right)$ be generated by Algorithm~\ref{alg:EG} for any sequence of functions $\lbrace f_t \rbrace_{t=1}^T$ for which Assumption~\ref{as:strongconcon}-\ref{as:lipgrad}, and \ref{as:liphess} hold true. Then, we have:
\begin{enumerate}[label=(\alph*)]
		\item Under the availability of the stochastic zeroth-order oracle, with
		\begin{align}
		\nu_{X[Y]}=\eta_t^2\left(d_{X[Y]}+3\right)^{-\frac{3}{2}}\quad \eta_t=\frac{4T^{-\alpha}}{\mu} \quad m_t^{X[Y]}=\frac{\left(d_{X[Y]}+5\right)}{\eta_t^2} \quad \alpha=\frac{1}{4}, \label{eq:statepsdeletamtchoicezero1}
		\end{align}
		we obtain
		\begin{align}\label{eq:rsspboundzero1}
		\mathfrak{R}_{SSP}\leq \mathcal{O}\left(\left(\sigma+1\right)T^\frac{3}{4}\right). 
		\end{align}
		Hence, the total number of calls to the stochastic zeroth-order oracle is $\mathcal{O}\left(\left(d_X+d_Y\right)T^\frac{5}{2}\right)$.\\
		Furthermore, by choosing
			\begin{align}
		\nu_{X[Y]}=\eta_t^4\left(d{X[Y]}+3\right)^{-\frac{3}{2}}\quad \eta_t=\frac{4T^{-\alpha}}{\mu} \quad m_t^{X[Y]}=\frac{\left(d_{X[Y]}+5\right)}{\eta_t^4}\quad \alpha=\frac{1}{4}, \label{eq:statepsdeletamtchoicezero2}
		\end{align}
		we obtain,
		\begin{align}
		\mathfrak{R}_{SSP}\leq \mathcal{O}\left(\sigma T^\frac{1}{2}+T^\frac{3}{4}\right). \label{eq:rsspboundzero2} 
		\end{align}
		Hence, the total number of calls to the stochastic zeroth-order  oracle is $\mathcal{O}\left(\left(d_X+d_Y\right)T^3\right)$.
		\item Under the availability of the stochastic first-order oracle, with
		\begin{align}
	    \eta_t=\frac{4T^{-\alpha}}{\mu} \quad m_t=\frac{1}{\eta_t^2}\quad \alpha=\frac{1}{4}, \label{eq:statepsdeletamtchoicefirst1}
		\end{align}
		we obtain,
		\begin{align}
		\mathfrak{R}_{SSP}\leq \mathcal{O}\left(\left(\sigma+1\right)T^\frac{3}{4}\right). \label{eq:rsspboundfirst1}
		\end{align} 
		Furthermore, by choosing,
		\begin{align}
		\eta_t=\frac{4T^{-\alpha}}{\mu} \quad m_t=\frac{1}{\eta_t^4}\quad \alpha=\frac{1}{4}, \label{eq:statepsdeletamtchoicefirst2}
		\end{align}
		we obtain,
		\begin{align}
		\mathfrak{R}_{SSP}\leq \mathcal{O}\left(\sigma T^\frac{1}{2}+T^\frac{3}{4}\right). \label{eq:rsspboundfirst2} 
		\end{align} 
	\end{enumerate}
\end{theorem}
\begin{remark}
    In Theorem~\ref{th:rsspcor}, better bounds for $\mathfrak{R}_{SSP}$ are achieved in \eqref{eq:rsspboundzero2}, and  \eqref{eq:rsspboundfirst2} compared to \eqref{eq:rsspboundzero1}, and \eqref{eq:rsspboundfirst1}. This improvement comes at a price of larger mini-batch size to estimate gradient, i.e., mini-batch size is of the order of $ t^\frac{2}{3}$ in the former case compared to $t^\frac{1}{2}$ in the later. 
\end{remark}
Before proving Theorem~\ref{th:rsspcor}, we first provide a high-level sketch. First, in Lemma~\ref{lm:BTL}, we show that if one observes the function before playing, i.e. if we could choose point $\left(x_{t+1},y_{t+1}\right)$ for function $f_t$, then $\mathfrak{R}_{SSP}$ is bounded by the sum of two terms: the first term, $L\sum_{t=1}^{T}\left(\norm{u^*_t - u^*_{t+1}}+\norm{v^*_t - v^*_{t+1}}\right)$, measures the closeness of the saddle points of $\calJ_t\left(x,y\right)$ over consecutive time steps, and the second term, $L\sum_{t=1}^{T}\left(\norm{v^*_{t+1} - y_{t+1}}+\norm{u^*_{t+1}-x_{t+1}}\right)$, measures how far the iterate at time $t+1$ is from the saddle point of $\calJ_{t}\left(x,y\right)$. To bound the first term, we leverage Lemma 2 of \cite{2018arXiv180608301R}, and show in Lemma~\ref{lm:optimachange} that at any time $t$, the saddle point of $\calJ_t\left(x,y\right)$ is close to the saddle point of $\calJ_{t-1}\left(x,y\right)$. To bound the second term, we first define \eqref{eq:sppstatic} and bound an auxiliary regret $\mathfrak{R}_{SPP}$ in Lemma~\ref{lm:statpathreg}. $\mathfrak{R}_{SPP}$ measures the squared distance of the iterate at time $t$ to the saddle point of $\calJ_{t}\left(x,y\right)$. Combining the above two results, and choosing the tuning parameters appropriately proves Theorem~\ref{th:rsspcor}. 
\begin{lemma}\label{lm:BTL} Under Assumption~\ref{as:lip}, we have
	\begin{align*}
	\abs{\sum_{t=1}^{T}f_t\left(x_{t+1},y_{t+1}\right) -  \sum_{t=1}^{T}f_t\left(u^*_{T+1},v^*_{T+1}\right)}
	\leq & L\sum_{t=1}^{T}\left(\norm{u^*_t - u^*_{t+1}}+\norm{v^*_t - v^*_{t+1}}+\norm{v^*_{t+1} - y_{t+1}}\right.\\
	&\left.+\norm{u^*_{t+1}-x_{t+1}}\right).\numberthis \label{eq:btl}
	\end{align*}
\end{lemma}
\begin{proof}[Proof of Lemma~\ref{lm:BTL}]
	We prove the following two inequalities by induction which implies \eqref{eq:btl}:
	\begin{align*}
	\sum_{t=1}^{T}f_t\left(x_{t+1},y_{t+1}\right) -  \sum_{t=1}^{T}f_t\left(u^*_{T+1},v^*_{T+1}\right)
	\leq & L\sum_{t=1}^{T}\left(\norm{v^*_t - v^*_{t+1}}+\norm{v^*_{t+1} - y_{t+1}}
	+\norm{u^*_{t+1}-x_{t+1}}\right)\\
	\sum_{t=1}^{T}f_t\left(x_{t+1},y_{t+1}\right) -  \sum_{t=1}^{T}f_t\left(u^*_{T+1},v^*_{T+1}\right)
	\geq & -L\sum_{t=1}^{T}\left(\norm{u^*_t - u^*_{t+1}}+\norm{v^*_{t+1} - y_{t+1}}
	+\norm{u^*_{t+1}-x_{t+1}}\right).
	\end{align*}
    \begin{align*}
       &\text{Base Case:}\quad f_1(u^*_2,v^*_2) \geq f_1(x_2,y_2) - L(\norm{v^*_1 - v^*_2}+\norm{v^*_2 - y_2}+\norm{u^*_2-x_2})\\
       &\text{The base case is true by Assumption~\ref{as:lip}.}\\ 
       &\text{Inductive Assumption:}\\&\sum_{t=1}^{T-1}f_t(u^*_T,v^*_T) \geq
       \sum_{t=1}^{T-1}f_t(x_{t+1},y_{t+1}) - L\sum_{t=1}^{T-1}\left(\norm{v^*_t - v^*_{t+1}}+\norm{v^*_{t+1} - y_{t+1}}+\norm{u^*_{t+1}-x_{t+1}}\right).
    \end{align*}
    So, using Assumption~\ref{as:lip}, and the inductive assumption, we get
    \begin{align*}
        \sum_{t=1}^{T}f_t(u^*_{T+1},v^*_{T+1})
        & \geq \sum_{t=1}^{T-1}f_t(u^*_{T+1},v^*_{T})+ f_T(u^*_{T+1},v^*_{T})\\
        & \geq \sum_{t=1}^{T-1}f_t(u^*_{T},v^*_{T}) + f_T(u^*_{T+1},v^*_{T})\\
        & \geq \sum_{t=1}^{T}f_t(x_{t+1},y_{t+1})  - L\sum_{t=1}^{T-1}\left(\norm{v^*_t - v^*_{t+1}}+\norm{v^*_{t+1} - y_{t+1}}+\norm{u^*_{t+1}-x_{t+1}}\right))\\
        &\quad + f_T(u^*_{T+1},v^*_{T}) - f_T(x_{T+1},y_{T+1})\\
        & \geq \sum_{t=1}^{T}f_t(x_{t+1},y_{t+1})- L\sum_{t=1}^{T-1}\left(\norm{v^*_t - v^*_{t+1}}+\norm{v^*_{t+1} - y_{t+1}}+\norm{u^*_{t+1}-x_{t+1}}\right) \\
        & \quad - L\left(\norm{v^*_T - y_{T+1}}+\norm{u^*_{T+1}-x_{T+1}}\right)\\
        & = \sum_{t=1}^{T}f_t(x_{t+1},y_{t+1})- L\sum_{t=1}^{T}\left(\norm{v^*_t - v^*_{t+1}}+\norm{v^*_{t+1} - y_{t+1}}+\norm{u^*_{t+1}-x_{t+1}}\right). 
    \end{align*}
    Using a similar approach, we get the result for another direction, thereby proving the required statement.
\end{proof}

In the following, we prove a bound for a regret which measures how far the iterates $\left(x_t,y_t\right)$ are from the saddle point $\left(u_{t+1}^*,y_{t+1}^*\right)$ of the sum of all the functions until corresponding time $t$. To be precise, we consider the regret
\begin{align}\label{eq:sppstatic}
\mathfrak{R}_{SPP}=\expec{\sum_{t=1}^{T}r_{t}^s}:=\sum_{t=1}^{T}\expec{\left(\|x_{t-1}-u_{t}^*\|^2+\|y_{t-1}-v_{t}^*\|^2\right)}.
\end{align}
Later we use this bound to bound $\mathfrak{R}_{SSP}$ in Theorem~\ref{th:rsspcor}.  
\begin{lemma}\cite{2018arXiv180608301R}\label{lm:optimachange}
Let $\lbrace f_t\rbrace_{t=1}^T$ be a sequence of functions for which Assumptions~\ref{as:strongconcon}--\ref{as:lip} are true. Let $\calJ_t\left(x,y\right)$, and $\left(u_t^*,v_t^*\right)$ be as in \eqref{eq:smoothedfunction}--\eqref{eq:smoothsaddle}. Then,
	\begin{align}
	\|u_{t}^*-u_{t+1}^*\|+\|v_t^*-v_{t+1}^*\|\leq \frac{4L}{\mu t}.
	\end{align}
\end{lemma}
\begin{lemma} \label{lm:statpathreg}
	Let $\left(x_t,y_t\right)$ be generated by Algorithm~\ref{alg:EG} for any sequence of functions $\lbrace f_t \rbrace_{t=1}^T$ for which Assumption~\ref{as:strongconcon}-\ref{as:lipgrad}, and \ref{as:liphess} hold true. Then,
\begin{enumerate}[label=(\alph*)]
		\item Under the availability of the stochastic zeroth-order oracle, choosing $\eta_t$, $\nu_{X[Y]}$, $\epsilon$, $\delta$, $m_t^{X[Y]}$, and $\alpha$ as in \eqref{eq:statepsdeletamtchoicezero1},
		we obtain,
		\begin{align}
		\mathfrak{R}_{SPP}\leq \mathcal{O}\left(\left(\sigma^2+1\right)\sqrt{T}\right). 
		\label{eq:rsppboundzero1}
		\end{align}
		Furthermore, by choosing $\eta_t$, $\nu_{X[Y]}$, $\epsilon$, $\delta$, and $m_t^{X[Y]}$ as in \eqref{eq:statepsdeletamtchoicezero2}, and $\alpha=\frac{1}{6}$,
		we obtain,
		\begin{align}
		\mathfrak{R}_{SPP}\leq \mathcal{O}\left(\left(\sigma^2+1\right)T^\frac{1}{3}\right). \label{eq:rsppboundzero2} 
		\end{align}
		\item Under the availability of the stochastic first-order oracle, choosing $\eta_t$, $\epsilon$, $\delta$, $m_t$, and $\alpha$ as in \eqref{eq:statepsdeletamtchoicefirst1},
		we obtain,
		\begin{align}
		\mathfrak{R}_{SPP}\leq \mathcal{O}\left(\left(\sigma^2+1\right)\sqrt{T}\right). \label{eq:rsppboundfirst1}
		\end{align} 
		Furthermore, with $\eta_t$, $\epsilon$, $\delta$, and $m_t$ as in \eqref{eq:statepsdeletamtchoicefirst2}, and $\alpha=\frac{1}{6}$,
		we obtain,
		\begin{align}
		\mathfrak{R}_{SPP}\leq \mathcal{O}\left(\left(\sigma^2+1\right)T^\frac{1}{3}\right). \label{eq:rsppboundfirst2} 
		\end{align} 
	\end{enumerate}
\end{lemma}
\begin{proof}[Proof of Lemma~\ref{lm:statpathreg}]
Let $r_{t}^s=\|x_{t-1}-u_{t}^*\|^2+\|y_{t-1}-v_{t}^*\|^2$. Let $x_t$, and $\hx_t$ be defined as in Lemma~\ref{lm:stocherrorppeg}. Then using Lemma~\ref{lm:optimachange}, for $\epsilon, \delta>0$, we get
	\begin{align*}
	&r_{t+1}^s=\|x_{t}-u_{t+1}^*\|^2+\|y_{t}-v_{t+1}^*\|^2\\
	\leq & \left(1+\frac{1}{\epsilon}\right)\left(\|x_{t}-u_{t}^*\|^2+\|y_{t}-v_{t}^*\|^2\right)+\left(1+\epsilon\right)\left(\|u_{t+1}^*-u_{t}^*\|^2+\|v_{t+1}^*-v_{t}^*\|^2\right)\\
	\leq & \left(1+\frac{1}{\epsilon}\right)\left(1+\frac{1}{\delta}\right)\left(\|\hx_{t}-u_{t}^*\|^2+\|\hy_{t}-v_{t}^*\|^2\right)\\
	+&\left(1+\frac{1}{\epsilon}\right)\left(1+\delta\right)\left(\|x_{t}-\hx_{t}\|^2+\|y_{t}-\hy_{t}\|^2\right)+\left(1+\epsilon\right)\frac{16L^2}{\mu^2t^2}.
	\end{align*}
	Let $\rho=\frac{1}{1+\eta_t \mu}$. Then using Theorem 2 of \cite{mokhtari2019unified}, and Lemma~\ref{lm:stocherrorppeg}, we get
	\begin{align*}
	\|\hx_{t}-u_{t}^*\|^2+\|\hy_{t}-v_{t}^*\|^2\leq & \rho\left(\|x_{t-1}-u_{t}^*\|^2+\|y_{t-1}-v_{t}^*\|^2\right) \numberthis \label{eq:exgradupdatestat}\\
	\expec{\|\hx_{t}-x_{t}\|^2+\|\hy_{t}-y_{t}\|^2|\calF_{t-1}}\leq & e_{0,t+1,X}+e_{0,t+1,Y}.
	\end{align*}
	Hence, we have
	\begin{align*}
	\expec{r_{t+1}^s|\calF_{t-1}}\leq &
	qr_t^s
	+\left(1+\frac{1}{\epsilon}\right)\left(1+\delta\right)\left(e_{0,t+1,X}+e_{0,t+1,Y}\right)
	+\left(1+\epsilon\right)\frac{16L^2}{\mu^2t^2}, \numberthis \label{eq:relwt1wtstat}
	\end{align*}
	where $q=\left(1+\frac{1}{\epsilon}\right)\left(1+\frac{1}{\delta}\right)\rho$. We'll choose $\eta_t$ to ensure $q<1$. From \eqref{eq:relwt1wtstat} we get,
	\begin{align*}
	\expec{r_{t+1}^s|\calF_{t-1}}\leq &
	q^{t}r_1^s
	+\frac{\left(e_{0,t+1,X}+e_{0,t+1,Y}\right)\left(1+\frac{1}{\epsilon}\right)\left(1+\delta\right)}{1-q} 
	+\left(1+\epsilon\right)\sum_{j=1}^{t}\frac{q^{t-j}}{j^2}. \numberthis \label{eq:relwt1w0stat}
	\end{align*}
	Summing both sides of \eqref{eq:relwt1w0stat} from $t=1$ to $T$ we obtain,
	\begin{align*}
	\sum_{t=1}^{T}\expec{r_{t+1}^s|\calF_{t-1}}\leq &
	\frac{q}{1-q}r_1^s
	+\frac{\left(e_{0,t+1,X}+e_{0,t+1,Y}\right)\left(1+\frac{1}{\epsilon}\right)\left(1+\delta\right)}{1-q}T
	+\frac{1+\epsilon}{1-q}\sum_{t=1}^{T}\frac{1}{t^2}. 
	\end{align*}
	Now, set $\epsilon=\delta=T^\alpha$. We now handle the zeroth-order and first-order setting separately:
	\begin{enumerate}[label=(\alph*)]
	    \item To ensure $q<1$ we choose, $\left(1+\frac{1}{\epsilon}\right)\left(1+\frac{1}{\delta}\right)\frac{1}{1+\eta_t \mu}<1$, i.e., $\eta_t>\frac{\frac{1}{\epsilon}+\frac{1}{\delta}+\frac{1}{\epsilon\delta}}{\mu}$. Choosing, $\nu_{X[Y]}$, $\eta_t$, $m_t^{X[Y]}$, and $\alpha$ as in \eqref{eq:statepsdeletamtchoicezero1}, we have the following set of (in)equalities:
	\begin{align*}
	q=\frac{\left(1+T^{-\alpha}\right)^2}{1+3T^{-\alpha}} \qquad & \qquad\frac{q}{1-q}=\frac{\left(1+T^{\alpha}\right)^2}{T^{\alpha}-1} \\
	\frac{\left(e_{0,t+1,X}+e_{0,t+1,Y}\right)\left(1+\frac{1}{\epsilon}\right)\left(1+\delta\right)}{1-q}T &\leq \frac{a_0\left(\sigma\right)^2T^{1-4\alpha }\left(T^{\alpha}+3\right)\left(1+T^{\alpha}\right)\left(1+T^\alpha \right)}{\mu^4\left(T^{\alpha}-1\right)}\\
	\frac{1+\epsilon}{1-q}&=\frac{T^{\alpha}\left(1+T^\alpha \right)\left(T^{\alpha}+3\right)}{T^{\alpha}-1}
	\end{align*}
	Hence, we obtain
	\begin{align*}
	\mathfrak{R}_{SPP}=\expec{\sum_{t=1}^{T}r_{t}^s}\leq \mathcal{O}\left(\left(\sigma^2+1\right)\sqrt{T}\right).
	\end{align*}
Following the similar approach and by choosing $\nu_{X[Y]}$, $\eta_t$, $m_t^{X[Y]}$, and $\alpha$ as in \eqref{eq:statepsdeletamtchoicezero2}, we get \eqref{eq:rsppboundzero2}.
	\item Choosing, $\eta_t$, $m_t$, and $\alpha$ as in \eqref{eq:statepsdeletamtchoicefirst1}, and \eqref{eq:statepsdeletamtchoicefirst2}, we get \eqref{eq:rsppboundfirst1}, and \eqref{eq:rsppboundfirst2} respectively.
	\end{enumerate}
\end{proof}

\begin{proof}[Proof of Theorem~\ref{th:rsspcor}]
Using Assumption~\ref{as:lip}, Lemma~\ref{lm:BTL}, and Lemma~\ref{lm:optimachange} we have,
\begin{align*}
	&\abs{\sum_{t=1}^{T}f_t\left(x_{t},y_{t}\right) -  \sum_{t=1}^{T}f_t\left(u^*_{T+1},v^*_{T+1}\right)}\\
	\leq & \abs{\sum_{t=1}^{T}f_t\left(x_{t},y_{t}\right) -  \sum_{t=1}^{T}f_t\left(x_{t+1},y_{t+1}\right)}+\abs{\sum_{t=1}^{T}f_t\left(x_{t+1},y_{t+1}\right) -  \sum_{t=1}^{T}f_t\left(u^*_{T+1},v^*_{T+1}\right)}\\
	\leq & L\sum_{t=1}^{T}\left(\|x_t-x_{t+1}\|+\|y_t-y_{t+1}\|\right)+ L\sum_{t=1}^{T}\left(\norm{u^*_t - u^*_{t+1}}+\norm{v^*_t - v^*_{t+1}}+\norm{v^*_{t+1} - y_{t+1}}\right.\\
	&\left.+\norm{u^*_{t+1}-x_{t+1}}\right)\\
	\leq & 2L\sum_{t=1}^{T}\left(\|x_t-x_{t+1}\|+\|y_t-y_{t+1}\|\right)+ \sum_{t=1}^{T}\frac{4L^2}{\mu t}+L\sum_{t=1}^{T}\left(\norm{v^*_{t+1} - y_{t}}+\norm{u^*_{t+1}-x_{t}}\right)\\
	\leq & 2\eta_t L^2T+\sum_{t=1}^{T}\frac{4L^2}{\mu t}+ L\sum_{t=1}^{T}\left(\norm{u^*_{t+1}-x_{t}}+\norm{v^*_{t+1} - y_{t}}\right).
	\end{align*}
	Squaring both sides we get,
\begin{align*}
\mathfrak{R}_{SSP}^2\leq & 12\eta_t^2 L^4T^2+3\left(\sum_{t=1}^{T}\frac{4L^2}{\mu t}\right)^2+ L^2\left(\sum_{t=1}^{T}\left(\norm{u^*_{t+1}-x_{t}}+\norm{v^*_{t+1} - y_{t}}\right)\right)^2\\
\leq & 12\eta_t^2 L^4T^2+3\left(\sum_{t=1}^{T}\frac{4L^2}{\mu t}\right)^2+ TL^2\sum_{t=1}^{T}\left(\norm{u^*_{t+1}-x_{t}}+\norm{v^*_{t+1} - y_{t}}\right)^2\\
\leq & 12\eta_t^2 L^4T^2+3\left(\sum_{t=1}^{T}\frac{4L^2}{\mu t}\right)^2+ 2TL^2\sum_{t=1}^{T}\left(\norm{u^*_{t+1}-x_{t}}^2+\norm{v^*_{t+1} - y_{t}}^2\right)\\
\leq & 12\eta_t^2 L^4T^2+3\left(\sum_{t=1}^{T}\frac{4L^2}{\mu t}\right)^2+ 2TL^2\mathfrak{R}_{SPP}\\
\mathfrak{R}_{SSP}\leq & 4\eta_t L^2T+8\sum_{t=1}^{T}\frac{L^2}{\mu t}+ 2L\sqrt{T\mathfrak{R}_{SPP}}.
\end{align*}
Set, $\epsilon=\delta=T^\alpha$. We now prove the zeroth-order and first-order setting separately.
\begin{enumerate}[label=(\alph*)]
\item Choosing $\alpha=\frac{1}{4}$, $\eta_t$, $\nu_{X[Y]}$, and $m_t^{X[Y]}$ as in \eqref{eq:statepsdeletamtchoicezero1}, and \eqref{eq:statepsdeletamtchoicezero2} we get \eqref{eq:rsspboundzero1}, and \eqref{eq:rsspboundzero2} respectively.
\item Choosing $\alpha=\frac{1}{4}$, $\eta_t$, $\nu$, and $m_t$ as in \eqref{eq:statepsdeletamtchoicefirst1}, and \eqref{eq:statepsdeletamtchoicefirst2} we get \eqref{eq:rsspboundfirst1}, and \eqref{eq:rsspboundfirst2} respectively.
\end{enumerate}
\end{proof}
\subsection{Static Regret Bounds for Frank-Wolfe Method}\label{sec:statfuncvalfw}
We now proceed to analyze Frank-Wolfe algorithms for the online saddle-point optimization problem~\eqref{eq:nssaddle}. Recall that in our analysis of the extragradient method in Section~\ref{sec:staticegmethod}, it was assumed that the problem was unconstrained, i.e., $\mathcal{X}= \mathbb{R}^{d_X}$ and $\mathcal{Y} = \mathbb{R}^{d_Y}$. Our main motivation in this section is to relax this assumption and to handle the constrained case, i.e., $\mathcal{X}$ and $\mathcal{Y}$ could be closed convex sets subset of $\mathbb{R}^{d_X}$ and $\mathbb{R}^{d_X}$ respectively. We now present our main result on bounding the static regret of Frank-Wolfe algorithm.

\input{fw}
\begin{theorem}
	\label{th:static_ol_firstOrder}
	Let $\{ f_t\}_{t=1}^T$ be an arbitrary sequence of functions for which Assumptions~\ref{as:strongconcon}, \ref{as:lip}, and \ref{as:lipgrad} hold. Then, if $C_0\geq 0$, (where $C_0$ is as defined in Theorem~\ref{theo:offl_nonadapt}), we have the following:
	\begin{enumerate} [label=(\alph*)] 
	\item Under the availability of the stochastic zeroth-order oracle, choosing
	\begin{align} \label{eq:inductionzero}
    \gamma_t = \frac{1}{\sqrt{t}} \quad m_t^{X[Y]}=2\left(d_{X[Y]}+5\right)t \quad \nu_{X[Y]}=\frac{\sqrt{2}}{\left(d_{X[Y]}+3\right)^\frac{3}{2}\sqrt{t}}
\end{align}
	we obtain 
	\begin{align}\label{eq:rsspboundfwzero}
		\mathfrak{R}_{SSP} \leq \order((1+\sigma)T^{\frac{3}{4}}).
	\end{align}
Hence, the total number of calls to the stochastic zeroth-order oracle is $\order((d_X+d_Y)T^2)$.
	\item Under the availability of the stochastic first-order oracle, choosing 
	\begin{align} \label{eq:inductionfirst}
    \gamma_t = \frac{1}{\sqrt{t}} \quad m_t^{X[Y]}=m_t=t,
\end{align}
	we obtain
	\begin{align}\label{eq:rsspboundfwfirst}
	\mathfrak{R}_{SSP} \leq \order\left((1+\sigma)T^{\frac{3}{4}}\right).
	\end{align}
	\end{enumerate}
\end{theorem}

\begin{remark}
It is worth comparing the batch size $m_t$ and smoothing parameter $\nu$ to the offline setting. Notice in Theorem~\ref{theo:offl_nonadapt}, we required (with $T$ denoting the number of iterations) $m_t = \order(dT^2)$, $\nu_t = \order(T^{-1}d^{-\frac{3}{2}})$. In the online setting, picking $m_t = \order(dT)$, $\nu_t = \order(T^{-\frac{1}{2}}d^{-\frac{3}{2}})$, is sufficient to obtain the above mentioned static regret bound.
\end{remark}
Before we prove Theorem~\ref{th:static_ol_firstOrder}, we provide a proof-sketch. Similar to the unconstrained setting, we first use Lemma~\ref{lm:BTL} to show that if the decision maker could play $\left(x_{t+1},y_{t+1}\right)$ for $f_t$, then $\mathfrak{R}_{SSP}$, could be decomposed as sum of two terms (as in the unconstrained case). The first term is handled using Lemma~\ref{lm:optimachange}. Bounding the second term, is a bit more involved than the unconstrained case. Indeed, Frank-Wolfe method progresses by reducing the merit function at each iteration as shown in Theorem~\ref{theo:offl_nonadapt}, whereas the Extragradient method progresses by reducing the distance from the saddle point. So in Lemma~\ref{lm:induct_helper}, and \ref{lm:induction} we first show the following term, $\expec{\calJ_{t}(x_{t},u^*_{t+1}) - \calJ_{t}(v^*_{t+1},y_{t})}$, which we refer to as the online merit function,  decreases as $\mathcal{O}\left(1/\sqrt{t}\right)$. Then we will show, using strong convexity-strong concavity (Assumption~\ref{as:strongconcon}) of the function sequence that the second term is upper bounded by the sum of the square roots of the online merit functions up to a constant. Combining all the above steps and picking the tuning parameters appropriately proves Theorem~\ref{th:static_ol_firstOrder}.
\begin{lemma}\label{lm:induct_helper}
Under Assumption~\ref{as:strongconcon} and~\ref{as:lip}, if $$\calJ_{t-1}(x_{t-1},u^*_{t}) - \calJ_{t-1}(v^*_{t},y_{t-1}) \leq C(t-1)^{-\beta},$$ for some constant $C>0$ and $\beta> 0$, then we have
 \begin{align*}
    f_t(x_t,u^*_{t+1}) - f_t(v^*_{t+1},y_t) \leq L\sqrt{\frac{2C}{\mu}}(t-1)^{-\beta/2}+ \frac{4L^2}{\mu t} + L(D_\calX+D_\calY)\gamma_{t-1}. \numberthis\label{eq:ifinductrue}
\end{align*}
\end{lemma}
\begin{proof}[Proof of Lemma~\ref{lm:induct_helper}]
By Assumption~\ref{as:strongconcon}, we have
\begin{align*}
    \norm{x_{t-1}-v^*_t}+\norm{y_{t-1}-u^*_t}
    \leq \sqrt{\frac{2}{\mu(t-1)}\sum_{\tau=1}^{t-1}\left(f_\tau(x_{t-1},u^*_t) - f_\tau(v^*_t,y_{t-1}) \right)}
    \leq \sqrt{\frac{2C}{\mu}}(t-1)^{-\beta/2}. 
\end{align*}
By Lemma~\ref{lm:optimachange}, we also have
\begin{align*}
    \norm{v^*_{t+1}-v^*_t}+\norm{u^*_{t+1}-u^*_t} 
    \leq \frac{4L}{\mu t}.
\end{align*}
Furthermore,
\begin{align*}
    \norm{x_t-x_{t-1}}+\norm{y_t-y_{t-1}} 
    =&\norm{(1-\gamma_t)x_{t-1}+\gamma_t s_t^x - x_{t-1}} \norm{(1-\gamma_t)y_{t-1}+\gamma_t s_t^y - y_{t-1}}\\
    =&~ \gamma_{t-1} (\norm{x_{t-1}-s_t^x} + \norm{y_{t-1}-s_t^y})=(D_\calX+D_\calY)\gamma_{t-1}.
\end{align*}
Adding the above inequalities, using triangle inequality, and Assumption~\ref{as:lip}, we have \eqref{eq:ifinductrue}.
\end{proof}
\begin{lemma}
\label{lm:induction} 
Consider the setting of Theorem~\ref{th:static_ol_firstOrder}. Then, 
\begin{enumerate}[label=(\alph*)]
\item in the zeroth-order stochastic setting, choosing $\gamma_t$, $\nu_{X[Y]}$, and $m_t^{X[Y]}$ as in \eqref{eq:inductionzero}
we get
\begin{align*}
    \expec{\calJ_{t}(x_{t},u^*_{t+1}) - \calJ_{t}(v^*_{t+1},y_{t})} \leq C_5t^{-\frac{1}{2}} \qquad \forall t. 
\end{align*}
\item in the first-order stochastic setting, choosing choosing $\gamma_t$, and $m_t$ as in \eqref{eq:inductionfirst}
we get
\begin{align*}
    \expec{\calJ_{t}(x_{t},u^*_{t+1}) - \calJ_{t}(v^*_{t+1},y_{t})} \leq C_6t^{-\frac{1}{2}} \qquad \forall t. 
\end{align*}
\end{enumerate}
where 
\begin{align*}
C_5 =C_6+\frac{2}{C_0L_G}\left(2L^2+3L_{GX}^2+3L_{GY}^2\right) , \quad
C_6=\frac{L^2}{C_0 \mu}\left(8+\frac{2}{C_0}\right)+\frac{2C_1}{C_0}+\frac{2L(D_\calX+D_\calY)}{C_0}+\frac{2\sigma^2}{C_0L_G}. 
\end{align*}
\end{lemma}
\begin{proof}[Proof of Lemma~\ref{lm:induction}]
We provide the proof by induction for part (b). The proof for part (a) is very similar. It can be easily seen that the base case is true by Assumption~\ref{as:lip}.
\begin{align*}
   \text{Base Case t=1:}\quad &\calJ_1(x_1,u^*_2) - \calJ_1(v^*_{2},y_{1}) \leq C_6 \\
   \text{Inductive Assumption:} \quad &\calJ_{t-1}(x_{t-1},u^*_{t}) - \calJ_{t-1}(v^*_{t},y_{t-1}) \leq C_6t^{-\beta}
\end{align*}
We will prove the induction for $\beta = \frac{1}{2}$. Using Assumption~\ref{as:lip}, \eqref{eq:standard_form} in Appendix, Lemma~\ref{lm:optimachange}, and choosing $\gamma_{t}$ as in \eqref{eq:inductionfirst}, we get
\begin{align*}
    &~~~~\calJ_{t-1}(x_t,u^*_{t+1}) - \calJ_{t-1}(v^*_{t+1},y_t)\\
    &\leq \calJ_{t-1}(x_t,u^*_t) - \calJ_{t-1}(v^*_t,y_t) + L\left( \norm{v^*_{t+1}-v^*_{t}}+ \norm{u^*_{t+1}-u^*_{t}}\right) \\
    &\leq (1-C_0\gamma_{t-1})\left(\calJ_{t-1}(x_{t-1},u^*_t) -  \calJ_{t-1}(v^*_t,y_{t-1})\right) + C_1\gamma^2_{t-1} + \frac{1}{2L_G}(\norm{\Delta_{t-1}^x}^2 + \norm{\Delta_{t-1}^y}^2) + \frac{4L^2}{\mu t} \\
    & \leq C_6(t-1)^{-\beta} - C_6C_0(t-1)^{-\frac{1}{2}-\beta} + \frac{C_1}{t-1}+ \frac{4L^2}{\mu t} + \frac{1}{2L_G}(\norm{\Delta_{t-1}^x}^2 + \norm{\Delta_{t-1}^y}^2).
    \numberthis\label{eq:lminductionsubresa}
\end{align*}
Now using \eqref{eq:grad_est_error}, we get
\begin{align*}
    \expec{\calJ_{t-1}(x_t,u^*_{t+1}) - \calJ_{t-1}(v^*_{t+1},y_t)|\calF_{t-1}}\leq  C_6(t-1)^{-\beta}\left(1 - \frac{C_0}{\sqrt{t-1}}\right) + \frac{C_1}{t-1}+ \frac{4L^2}{\mu t}
    + \frac{\sigma^2}{m_{t-1}L_G}.
\end{align*}
Combining \eqref{eq:lminductionsubresa} with Lemma~\ref{lm:induct_helper}, and choosing $m_t$ as in \eqref{eq:inductionfirst}, we get
\begin{align*}
    &\expec{\calJ_{t}(x_t,u^*_{t+1}) - \calJ_{t}(v^*_{t+1},y_t)|\calF_{t}}
    \leq \frac{1}{t}\left(C_6(t-1)^{1-\beta}  +   C_1 -C_6C_0(t-1)^{\frac{1}{2}-\beta} +  \frac{4L^2}{\mu}+\frac{\sigma^2}{L_G}\right)\\
    & +\frac{1}{t}\left( L\sqrt{\frac{2C_6}{\mu}}(t-1)^{-\beta/2}+\frac{L(D_\calX+D_\calY)}{\sqrt{t-1}}\right)\\
    &\leq C_{6}t^{-\beta}  + \frac{1}{t} \left(C_1-C_{6 }C_0(t-1)^{\frac{1}{2}-\beta} + \frac{4L^2}{\mu} + L\sqrt{\frac{2C_{6 }}{\mu}}(t-1)^{-\beta/2}+\frac{L(D_\calX+D_\calY)}{\sqrt{t-1}}+\frac{\sigma^2}{L_G}\right) \\
    &\leq C_{6}t^{-\beta}.
\end{align*}
The last inequality follows from the fact that the second term is non-positive $\forall t$ when $\beta = \frac{1}{2}$.
\end{proof}
\vspace{0.2in}
\begin{proof}[Proof of~Theorem~\ref{th:static_ol_firstOrder}]
     Using Asusmption~\ref{as:lip}, Lemma~\ref{lm:optimachange}, Lemma~\ref{lm:BTL}, and choosing $\gamma_t$ as in \eqref{eq:inductionzero} we get
\begin{align*}
    &\abs{\sum_{t=1}^T f_t(x_t,y_t) - \sum_{t=1}^T f_t(v^*_{T+1},u^*_{T+1})}
    \leq\abs{\sum_{t=1}^T f_t(x_t,y_t) - \sum_{t=1}^T f_t(x_{t+1},y_{t+1})}\\
    +&L\sum_{t=1}^{T}\left(\norm{u^*_t - u^*_{t+1}}+\norm{v^*_t - v^*_{t+1}}+\norm{v^*_{t+1} - y_{t+1}}
	+\norm{u^*_{t+1}-x_{t+1}}\right)\\
    \leq &~L\sum_{t=1}^{T}\left(\norm{x_t-x_{t+1}}+ \norm{y_t-y_{t+1}}+\frac{4L}{\mu t}+\norm{u^*_{t+1} - x_{t+1}}+\norm{v^*_{t+1}-y_{t+1}}\right) \\
     \leq &~L\sum_{t=1}^{T}\left(2\gamma_t(D_\calX+D_\calY)+ \frac{4L}{\mu t}+\norm{u^*_{t+1} - x_{t}}+\norm{v^*_{t+1}-y_{t}}\right)\\
    \leq &~L\sum_{t=1}^{T}\left(\norm{u^*_{t+1} - x_{t}}+\norm{v^*_{t+1}-y_{t}}\right) + L(D_\calX+D_\calY)\sum_{t=1}^T\frac{1}{\sqrt{t}}+\frac{4L^2}{\mu}\log T\\
    \leq &~L\sum_{t=1}^{T}\left(\sqrt{\frac{2}{\mu} [\calJ_{t}(x_{t},u^*_{t+1})-\calJ_{t}(v^*_{t+1},y_{t})]}\right) + \order(\sqrt{T}).
\end{align*}
Taking conditional expectation on both sides, and using Jensen's inequality,
\begin{align*}
    \expec{\abs{\sum_{t=1}^T f_t(x_t,y_t) - \sum_{t=1}^T f_t(v^*_{T+1},u^*_{T+1})}}\leq \frac{\sqrt{2L^2}}{\sqrt{\mu}}\sum_{t=1}^{T}\sqrt{\expec{\calJ_{t}(x_{t},u^*_{t+1})-\calJ_{t}(v^*_{t+1},y_{t})}} + \order(\sqrt{T}).
\end{align*}
We now proceed diferently for the zeroth and first-order setting.
\begin{enumerate}[label=(\alph*)]
\item Choosing $\gamma_t$, $\nu_{X[Y]}$, and $m_t^{X[Y]}$ as in \eqref{eq:inductionzero}, and using Lemma~\ref{lm:induction}, we get
\begin{align*}
    \expec{\abs{\sum_{t=1}^T f_t(x_t,y_t) - \sum_{t=1}^T f_t(v^*_{T+1},u^*_{T+1})}}\leq \frac{\sqrt{2L^2}}{\sqrt{\mu}}\sum_{t=1}^{T}\frac{C_5}{t^\frac{1}{4}}+ \order(\sqrt{T})\leq \order\left(\left(1+\sigma\right)T^\frac{3}{4}\right).
\end{align*}
\item Choosing $\gamma_t$, and $m_t$ as in \eqref{eq:inductionfirst}, and using Lemma~\ref{lm:induction}, we get \eqref{eq:rsspboundfwfirst}.
\end{enumerate}
\end{proof}


%% file: dynamicregret.tex
\section{Dynamic Regret Bounds for Saddle-Point Optimization}\label{sec:dynamicsection}
In the nonstationary setting, as the saddle point of the functions change over time, a stronger notion of regret compared to the static regret is worth exploring. For example, in the motivating example in Section~\ref{sec:motiveg}, the players may be interested to be able to play the Nash equilibrium at each time point. In this case, it is required to evaluate an algorithm's performance by how close the generated points are, either in terms of function values or in terms of distance to the saddle-points of the functions, at each time-step instead of the saddle-point of the sum of all the functions. In Sections~\ref{sec:dynamicegmethod} and~\ref{sec:dynfuncvalfw}, we propose such notions of dynamic regret and provide sub-linear regret bounds. 

\subsection{Dynamic Regret Bounds for Extragradient Method}\label{sec:dynamicegmethod}
For the online extragradient algorithm in the unconstrained setting, we propose the following natural notion of dynamic regret.
\begin{definition} \label{def:dsppath}
Let $\lbrace f_t \rbrace_{t=1}^T$ be a sequence of functions satisfying Assumption~\ref{as:strongconcon}. For this class of functions, with $(x_{t}^*,y_{t}^*)$ as defined in~\eqref{eq:nssaddle}, the Dynamic Saddle-Point Path (DSPP) Regret is defined as 
\begin{align}
\mathfrak{R}_{DSPP}:=\sum_{t=1}^{T}\expec{r_t}=\sum_{t=1}^{T}\expec{\|x_{t}-x_{t}^*\|^2+\|y_{t}-y_{t}^*\|^2}. \label{eq:dsppath}
\end{align}
where the expectation is taken w.r.t the filtration generated by $\lbrace x_t,y_t \rbrace_1^T$.
\end{definition}
In the following theorem we state our regret bounds for $\mathfrak{R}_{DSPP}$ using Extra-gradient method.
\begin{theorem}\label{th:dynamicexgrad}
	Let $\lbrace f_t\rbrace_{t=1}^T$ be a sequence of functions for which Assumptions~\ref{as:strongconcon} and \ref{as:lip} hold and such that $\lbrace f_t\rbrace_{t=1}^T \in \mathcal{M}_T$. Let $\alpha,\beta >0$ be some constants. Then, for $V_T\leq T^\frac{\alpha}{\beta}$, running Algorithm~\ref{alg:EG}, we have:
	\begin{enumerate}[label=(\alph*)]
		\item Under the availability of the stochastic zeroth-order oracle, with
		\begin{align}
		\eta_t=\eta=\frac{4T^{-\alpha}V_T^\beta}{\mu} \quad  \nu_{X[Y]}=\eta^2\left(d_{X[Y]}+3\right)^{-\frac{3}{2}} \quad m_t^{X[Y]}=\frac{\left(d_{X[Y]}+5\right)}{\eta^{2}}, \label{eq:epsdeletamtchoicezero1}
		\end{align}
		we obtain,
		\begin{align}
		\mathfrak{R}_{DSPP}\leq \mathcal{O}\left(\sigma^2T^{1-{2\alpha}}V_T^{2\beta}+V_T^{1-2\beta}T^{2\alpha}\right).\label{eq:rdsppboundzero1}
		\end{align}
		Hence, the total number of calls to the stochastic zeroth-order oracle is $\mathcal{O}\left(\left(d_X+d_Y\right)T^{1+2\alpha}V_T^{-2\beta}\right)$. Furthermore, by choosing
 		\begin{align}
		\eta_t=\eta=\frac{4T^{-\alpha}V_T^\beta}{\mu} \quad  \nu_{X[Y]}=\eta^4\left(d_{X[Y]}+3\right)^{-\frac{3}{2}} \quad m_t^{X[Y]}=\frac{\left(d_{X[Y]}+5\right)}{\eta^{4}}, \label{eq:epsdeletamtchoicezero2}
		\end{align}
		we obtain,
		\begin{align}
		\mathfrak{R}_{DSPP}\leq \mathcal{O}\left(\sigma^2T^{1-{4\alpha}}V_T^{4\beta}+V_T^{1-2\beta}T^{2\alpha}\right).\label{eq:rdsppboundzero2}
		\end{align}
		Hence, the total number of calls to the stochastic zeroth-order  oracle is $\mathcal{O}\left(\left(d_X+d_Y\right)T^{1+4\alpha} V_T^{-4\beta}\right)$.
		\item Under the availability of the stochastic first-order oracle, with
		\begin{align}
		\eta_t=\eta=\frac{4T^{-\alpha}V_T^\beta}{\mu} \quad  m_t=\frac{1}{\eta^{2}}, \label{eq:epsdeletamtchoicefirst1}
		\end{align}
		we obtain,
		\begin{align}
		\mathfrak{R}_{DSPP}\leq \mathcal{O}\left(\sigma^2T^{1-{2\alpha}}V_T^{2\beta}+V_T^{1-2\beta}T^{2\alpha}\right).\label{eq:rdsppboundfirst1}
		\end{align}
		Furthermore, by choosing
		\begin{align}
		\eta_t=\eta=\frac{4T^{-\alpha}V_T^\beta}{\mu} \quad m_t=\frac{1}{\eta^{4}}, \label{eq:epsdeletamtchoicefirst2}
		\end{align}
		we obtain,
		\begin{align}
		\mathfrak{R}_{DSPP}\leq \mathcal{O}\left(\sigma^2T^{1-{4\alpha}}V_T^{4\beta}+V_T^{1-2\beta}T^{2\alpha}\right).\label{eq:rdsppboundfirst2}
		\end{align}
	\end{enumerate}
\end{theorem}

\begin{proof}[Proof of Theorem~\ref{th:dynamicexgrad}]
Let $r_{t}=\|x_{t}-x_{t}^*\|^2+\|y_{t}-y_{t}^*\|^2$ as in Definition~\ref{def:dsppath}. For $\epsilon, \delta>0$ 
	\begin{align*}
	&r_{t+1}=\|x_{t+1}-x_{t+1}^*\|^2+\|y_{t+1}-y_{t+1}^*\|^2\\
	\leq & \left(1+\frac{1}{\epsilon}\right)\left(\|x_{t+1}-x_{t}^*\|^2+\|y_{t+1}-y_{t}^*\|^2\right)+\left(1+\epsilon\right)\left(\|x_{t+1}^*-x_{t}^*\|^2+\|y_{t+1}^*-y_{t}^*\|^2\right)\\
	\leq & \left(1+\frac{1}{\epsilon}\right)\left(1+\frac{1}{\delta}\right)\left(\|\hx_{t+1}-x_{t}^*\|^2+\|\hy_{t+1}-y_{t}^*\|^2\right)\\
	+&\left(1+\frac{1}{\epsilon}\right)\left(1+\delta\right)\left(\|x_{t+1}-\hx_{t+1}\|^2+\|y_{t+1}-\hy_{t+1}\|^2\right)+\left(1+\epsilon\right)\left(\|x_{t+1}^*-x_{t}^*\|^2+\|y_{t+1}^*-y_{t}^*\|^2\right).
	\end{align*}
	We now invoke the following result from \cite{mokhtari2019unified}.
	\begin{lemma}[Theorem 2 in \cite{mokhtari2019unified}]\label{lm:th2mokh}
	Under Assumption~\ref{as:strongconcon}, for any $\eta_t > 0$, the iterates $\lbrace x_t, y_t\rbrace\geq 0$ generated by the proximal point method satisfy
	\begin{align*}
	\|x_{t+1}-x_{t}^*\|^2+\|y_{t+1}-y_{t}^*\|^2\leq & \rho\left(\|x_{t}-x_{t}^*\|^2+\|y_{t}-y_{t}^*\|^2\right),
	\end{align*}
	where $\rho=\frac{1}{1+\eta_t \mu}$.
	\end{lemma}
	Using Lemma~\ref{lm:stocherrorppeg}, and Lemma~\ref{lm:th2mokh}, we get
	\begin{align*}
	&\expec{\|\hx_{t+1}-x_{t+1}\|^2+\|\hy_{t+1}-y_{t+1}\|^2|\calF_t}\leq  \left(e_{0,t+1,X}+e_{0,t+1,Y}\right)\\
	&\|\hx_{t+1}-x_{t}^*\|^2+\|\hy_{t+1}-y_{t}^*\|^2\leq  \rho\left(\|x_{t}-x_{t}^*\|^2+\|y_{t}-y_{t}^*\|^2\right). \numberthis \label{eq:exgradupdate}
	\end{align*}
	Hence, we have
	\begin{align*}
	\expec{r_{t+1}|\calF_t}\leq &~
	qr_t +\left(1+\frac{1}{\epsilon}\right)\left(1+\delta\right)\left(e_{0,t+1,X}+e_{0,t+1,Y}\right)\\
	&~+\left(1+\epsilon\right)\left(\|x_{t+1}^*-x_{t}^*\|^2+\|y_{t+1}^*-y_{t}^*\|^2\right), \numberthis \label{eq:relwt1wt}
	\end{align*}
	where $q=\left(1+\frac{1}{\epsilon}\right)\left(1+\frac{1}{\delta}\right)\rho$. We'll choose $\eta$ to ensure $q<1$. From \eqref{eq:relwt1wt} we get,
	\begin{align*}
	\expec{r_{t+1}|\calF_t}\leq &
	q^{t+1}r_0
	+\frac{\left(e_{0,t+1,X}+e_{0,t+1,Y}\right)\left(1+\frac{1}{\epsilon}\right)\left(1+\delta\right)}{1-q} 
	+\left(1+\epsilon\right)\sum_{j=0}^{t}q^{t-j}d_j, \numberthis \label{eq:relwt1w0}
	\end{align*}
	where $d_j=\left(\|x_{j+1}^*-x_{j}^*\|^2+\|y_{j+1}^*-y_{j}^*\|^2\right)$. Summing both sides of \eqref{eq:relwt1w0} from $t=0$ to $T-1$ we get,
	\begin{align*}
	\sum_{t=0}^{T-1}\expec{r_{t+1}|\calF_t}\leq &
	\frac{q}{1-q}r_0
	+\frac{\left(e_{0,t+1,X}+e_{0,t+1,Y}\right)\left(1+\frac{1}{\epsilon}\right)\left(1+\delta\right)}{1-q}T
	+\frac{1+\epsilon}{1-q}\sum_{t=0}^{T-1}d_t. 
	\end{align*}
	To ensure $q<1$ we choose, $\left(1+\frac{1}{\epsilon}\right)\left(1+\frac{1}{\delta}\right)\frac{1}{1+\eta \mu}<1$, i.e., $\eta>\frac{\frac{1}{\epsilon}+\frac{1}{\delta}+\frac{1}{\epsilon\delta}}{\mu}$. Now, set $\epsilon=\delta=T^\alpha V_T^{-\beta}$. To complete the proof, we handle the zeroth-order and first-order setting separately below:
	\begin{enumerate}[label=(\alph*)]
	    \item Choosing $\eta$, $\nu_{X[Y]}$, and $m_t^{X[Y]}$ as in \eqref{eq:epsdeletamtchoicezero1}, we have the following set of (in)equalities:
	\begin{align*}
	q=\frac{\left(1+T^{-\alpha}V_T^\beta\right)^2}{1+3T^{-\alpha}V_T^\beta}, \quad&\quad \frac{q}{1-q}=\frac{\left(1+T^{\alpha}V_T^{-\beta}\right)^2}{T^{\alpha}V_T^{-\beta}-1}, \\
	\frac{\left(e_{0,t+1,X}+e_{0,t+1,Y}\right)\left(1+\frac{1}{\epsilon}\right)\left(1+\delta\right)}{1-q}T & \leq \frac{a_0\left(\sigma\right)^2T^{1-4\alpha }V_T^{3\beta}\left(T^{\alpha}+4V_T^{\beta}\right)\left(V_T^{\beta}+T^{\alpha}\right)^2}{\mu^4\left(2T^{\alpha}V_T^{\beta}-V_T^{2\beta}\right)},\\
	\frac{1+\epsilon}{1-q}&=\frac{T^{\alpha}V_T^{-\beta}\left(1+T^\alpha V_T^{-\beta}\right)\left(T^{\alpha}V_T^{-\beta}+4\right)}{T^{\alpha}V_T^{-\beta}-1}.
	\end{align*}
	Hence, we have
	\begin{align*}
	\mathfrak{R}_{DSPP}=\expec{\sum_{t=1}^{T}r_{t}}\leq \mathcal{O}\left(\sigma^2T^{1-{2\alpha}}V_T^{2\beta}+V_T^{1-2\beta}T^{2\alpha}\right).
	\end{align*}
	Choosing $\eta$, $\nu_{X[Y]}$, and $m_t^{X[Y]}$ as in \eqref{eq:epsdeletamtchoicezero2}, we get \eqref{eq:rdsppboundzero2}.
	\item The proof is very similar to part (a). Choosing $\eta$, $\nu_{X[Y]}$, and $m_t^{X[Y]}$ as in \eqref{eq:epsdeletamtchoicefirst1}, and \eqref{eq:epsdeletamtchoicefirst2}, we get \eqref{eq:rdsppboundfirst1}, and \eqref{eq:rdsppboundfirst2} respectively. 
	\end{enumerate}
\end{proof}
\begin{remark}
	If $V_T\leq T$, choosing $\alpha=\beta=1/4$ in \eqref{eq:rdsppboundzero1} we get $\mathfrak{R}_{DSPP}\leq \mathcal{O}\left(\sqrt{TV_T}\left(\sigma^2+1\right)\right)$, and by choosing $\alpha=\beta=
	1/6$ in \eqref{eq:rdsppboundzero2} we get $\mathfrak{R}_{DSPP}\leq \mathcal{O}\left(\left(TV_T\right)^\frac{1}{3}\left(\sigma^2+1\right)\right)$. The price we pay for improved bounds is a larger mini-batch size $m_t$ for the gradient estimator which is $\sqrt{T/V_T}$ in the former case compared to $\left(T/V_T\right)^{2/3}$ in the later.
\end{remark}
\begin{remark}
It is possible to consider an alternate regret notion in the unconstrained setting in terms of function-value sub-optimality as well. Specifically, consider the following notion of regret:
\begin{definition} \label{def:dspf}
Let $\lbrace f_t\rbrace_{t=1}^{T}$ be a sequence of functions satisfying Assumption~\ref{as:strongconcon}. For this class of functions, with $(x_{t}^*,y_{t}^*)$ as defined in~\eqref{eq:nssaddle}, the Dynamic Saddle-Point Function-value (DSPF) Regret is defined as 
\begin{align}
\mathfrak{R}_{DSPF}=\expec{\sum_{t=1}^{T}\abs{f_t\left(x_t,y_t\right)-f_t\left(x_t^*,y_t^*\right)}}. \label{eq:dspf}
\end{align}
where the expectation is taken w.r.t the filtration generated by $\lbrace x_t,y_t \rbrace_1^T$.
\end{definition}
We now show that regret bounds for $\mathfrak{R}_{DSPF}$ could be obtained by the corresponding bounds for $\mathfrak{R}_{DSPP}$ obtained in Theorem~\ref{th:dynamicexgrad}. First note that using Assumption~\ref{as:lip}, we obtain
    \begin{align*}
       \expec{\left( \sum_{t=1}^T\abs{f_t\left(x_t,y_t\right)-f_t\left(x_t^*,y_t^*\right)}\right)^2}\leq & \expec{\left( L\sum_{t=1}^T\left(\|x_t-x_t^*\|+\|y_t-y_t^*\|\right)\right)^2}\\
       \leq & 2TL^2\expec{\sum_{t=1}^T\left(\|x_t-x_t^*\|^2+\|y_t-y_t^*\|^2\right)}\\
       \leq & 2TL^2\mathfrak{R}_{DSPP}.
       \end{align*}
       Hence, for the notion of regret in Definition~\ref{def:dspf}, we have
       \begin{align*}
       \expec{\sum_{t=1}^T\abs{f_t\left(x_t,y_t\right)-f_t\left(x_t^*,y_t^*\right)}}\leq & \sqrt{\expec{\left( \sum_{t=1}^T\abs{f_t\left(x_t,y_t\right)-f_t\left(x_t^*,y_t^*\right)}\right)^2}} \leq 2L\sqrt{T\mathfrak{R}_{DSPP}}.
    \end{align*}
\end{remark}
\subsection{Dynamic Regret Bounds for Frank-Wolfe Method}\label{sec:dynfuncvalfw}
For the constrained setting, using Frank-Wolfe algorithm, we propose the following notion of regret based on function values.
\begin{definition} \label{def:dsp2}
Let $\lbrace f_t \rbrace_{t=1}^T$ be a sequence of functions satisfying Assumption~\ref{as:strongconcon}. For this class of functions, with $(x_{t}^*,y_{t}^*)$ as defined in~\eqref{eq:nssaddle}, the Dynamic Saddle-Point Merit (DSPM) Regret is defined as 
\begin{align}
\mathfrak{R}_{DSPM}=\expec{\sum_{t=1}^{T}\left(f_t\left(x_t,y_t\right)-f_t\left(x_t^*,y_t^*\right)\right)^2}.
\end{align}
\end{definition}
In the following theorem we state our regret bounds for $\mathfrak{R}_{DSPM}$ using Frank-Wolfe method.
\begin{theorem}\label{thm:fwdynamicfuncreg}
Let $\lbrace f_t \rbrace_{t=1}^T$ be a sequence of function for which Assumptions~\ref{as:lip}, \ref{as:lipgrad}, \ref{as:crosslipgrad}, and \ref{as:strongconcon} hold. Furthermore, let $\lbrace f_t \rbrace_{t=1}^T\in \mathcal{M}_T~\cap~\mathcal{D}_T$. If $0\leq C_0\leq 1$, where $C_0=1-\frac{\sqrt{2}}{\delta_\mu}\max\left\lbrace\frac{D_\calX L_{XY}}{\sqrt{\mu_Y}},\frac{D_\calY L_{YX}}{\sqrt{\mu_X}}\right\rbrace$ and $\delta_\mu=\sqrt{\min\left(\mu_X\delta_Y^2,\mu_Y\delta_Y^2\right)}$, then running Algorithm~\ref{alg:FWdyn}, we have:
	\begin{enumerate}[label=(\alph*)]
	    \item Under the availability of the stochastic zeroth-order oracle, with
	\begin{align}
	\gamma_t=\gamma=\sqrt{\frac{\max\left(2W_T,V_T\right)}{T}} \quad m_t^{X[Y]}=T\left(d_{X[Y]}+5\right) \quad \nu_{X[Y]}=\frac{1}{\sqrt{T\left(d_{X[Y]}+3\right)^3}},\label{eq:gammamtnuchoicedspmzero}
	\end{align}
	we obtain, 
	\begin{align}
	\mathfrak{R}_{DSPM}\leq \mathcal{O}\left(\sqrt{T\max\left(2W_T,V_T\right)}\left(1+\frac{\sigma^2}{\max\left(2W_T,V_T\right)}\right)\right). \label{eq:rdspmboundzero}
	\end{align}
	Hence, the total number of calls to the stochastic zeroth-order oracle is $\mathcal{O}\left(\left(d_X+d_Y\right)T^2\right)$.
	\item Under the availability of the stochastic first-order oracle, with
	\begin{align}
	\gamma_t=\gamma=\sqrt{\frac{\max\left(2W_T,V_T\right)}{T}} \quad m_t=T, \label{eq:gammamtnuchoicedspmfirst}
	\end{align}
    we obtain
    \begin{align}
	\mathfrak{R}_{DSPM}\leq \mathcal{O}\left(\sqrt{T\max\left(2W_T,V_T\right)}\left(1+\frac{\sigma^2}{\max\left(2W_T,V_T\right)}\right)\right). \label{eq:rdspmboundfirst}
	\end{align}
	\end{enumerate}
\end{theorem}
\begin{proof}[Proof of Theorem~\ref{thm:fwdynamicfuncreg}]
First note that by Assumption~\ref{as:lip}, we have
	\begin{align*}
	\sum_{t=1}^{T}\left(f_t\left(x_t,y_t\right)-f_t\left(x_t^*,y_t^*\right)\right)^2
	\leq  2L^2 \sum_{t=1}^{T}\left(\|x_t-x_t^*\|^2+\|y_t-y_t^*\|^2\right).
	\end{align*}
Using Assumption~\ref{as:strongconcon} we have,
	\begin{align*}
	2L^2 \sum_{t=1}^{T}\left(\|x_t-x_t^*\|^2+\|y_t-y_t^*\|^2\right)
	\leq & \frac{4L^2}{\mu}\sum_{t=1}^{T}\left(f_{t}\left(x_t,y_t^*\right)-f_{t}\left(x_t^*,y_t^*\right)+f_{t}\left(x_t^*,y_t^*\right)-f_{t}\left(x_t^*,y_t\right)\right)\\
	\leq & \frac{4L^2}{\mu}\sum_{t=1}^{T}\left(f_{t}\left(x_t,y_t^*\right)-f_{t}\left(x_t^*,y_t\right)\right). \numberthis \label{eq:relfwdynamicreg}
	\end{align*}
	By Algorithm~\ref{alg:FWdyn} update \eqref{eq:alg2update}, we have 
	\begin{align}
	f_{t}\left(x_{t+1},y_{t}^*\right)-f_{t}\left(x_{t}^*,y_{t+1}\right)\leq \left(1-C_0\gamma\right)\left(f_{t}\left(x_{t},y_t^*\right)-f_{t}\left(x_t^*,y_{t}\right)\right)+\gamma^2C_1+\frac{\|\Delta_t^x\|^2+\|\Delta_t^y\|^2}{2L_G}. \label{eq:spfwupdate}
	\end{align}
Let $w_{t}=f_{t}\left(x_{t},y_{t}^*\right)-f_{t}\left(x_{t}^*,y_{t}\right)$. Then we have
	\begin{align*}
	w_{t+1}=&f_{t+1}\left(x_{t+1},y_{t+1}^*\right)-f_{t+1}\left(x_{t+1}^*,y_{t+1}\right)\\
	=&f_{t}\left(x_{t+1},y_{t}^*\right)-f_{t}\left(x_{t}^*,y_{t+1}\right)
	+f_{t+1}\left(x_{t+1},y_{t}^*\right)-f_{t}\left(x_{t+1},y_{t}^*\right)
	+f_{t}\left(x_{t}^*,y_{t+1}\right)-f_{t+1}\left(x_{t}^*,y_{t+1}\right)\\
	&~~+f_{t+1}\left(x_{t+1},y_{t+1}^*\right)-f_{t+1}\left(x_{t+1},y_{t}^*\right)
	+f_{t+1}\left(x_{t}^*,y_{t+1}\right)-f_{t+1}\left(x_{t+1}^*,y_{t+1}\right).  \numberthis\label{eq:dynfwwt1breakdown}
	\end{align*}
	We now proceed to handle zeroth-order and first-order setting differently. 
	\begin{enumerate}[label=(\alph*)]
    \item Using \eqref{eq:grad_est_error}, choosing $m_t^{X[Y]}$, and $\nu_{X[Y]}$ as in \eqref{eq:gammamtnuchoicedspmzero} we have
	\begin{align}
	    \expec{\|\Delta_t^{x[y]}\|^2}\leq \frac{2\left(L^2+\sigma^2\right)+\frac{3}{2}L_G^2}{T}.\label{eq:expectedgraderrormtnu}
	\end{align}
	Using \eqref{eq:expectedgraderrormtnu}, and Assumption~\ref{as:lip} we get,
	\begin{align*}
	\expec{w_{t+1}|\calF_t}\leq \left(1-C_0\gamma\right)w_t+\gamma^2C_1+\frac{4\left(L^2+\sigma^2\right)+3L_G^2}{2L_GT} +2a_t+Lb_t,
	\end{align*}
	where $a_t=\norm{f_t-f_{t+1}}:=\sup_{x,y\in\mathcal{X},\calY}\abs{f_t\left( x,y\right) -f_{t+1}\left( x,y\right)}$, and $b_t=\|x_t^*-x_{t+1}^*\|+\|y_t^*-y_{t+1}^*\|$. 
	Hence, using $r_0\leq 2LD$, where $D=\max\left(D_X,D_Y\right)$, we get 
	\begin{align*}
	\expec{w_{t+1}|\calF_t}\leq& ~2\left(1-C_0 \gamma\right)^{t+1}LD+\gamma^2C_1\sum_{j=0}^{t}\left(1-C_0 \gamma\right)^{j}+\frac{4\left(L^2+\sigma^2\right)+3L_G^2}{2C_0\gamma L_GT}\\
	&~+\sum_{j=0}^{t}\left(1-C_0 \gamma\right)^{t-j}\left(2a_j+Lb_j\right).
	\end{align*}
	Summing both sides from $t=0$, to $t=T-1$, we get
	\begin{align*}
	\sum_{t=0}^{T-1}\expec{w_{t+1}|\calF_t}\leq & 2LD\sum_{t=0}^{T-1}\left(1-C_0 \gamma\right)^{t+1}+\gamma^2C_1\sum_{t=0}^{T-1}\sum_{j=0}^{t}\left(1-C_0 \gamma\right)^{j}+\frac{4\left(L^2+\sigma^2\right)+3L_G^2}{2C_0\gamma L_G}\\
	+&\sum_{t=0}^{T-1}\sum_{j=0}^{t}\left(1-C_0 \gamma\right)^{t-j}\left(2a_j+Lb_j\right)\\
	\leq & \frac{2LD\left(1-C_0 \gamma\right)}{C_0 \gamma}+\frac{\gamma^2C_1T}{1-\left(1-C_0 \gamma\right)}+\frac{4\left(L^2+\sigma^2\right)+3L_G^2}{2C_0\gamma L_G}
	+\frac{1}{C_0 \gamma}\sum_{t=0}^{T-1}\left(2a_t+Lb_t\right)\\
	\leq & \frac{2LD}{C_0\gamma}+\frac{\gamma C_1T}{C_0}+\frac{4\left(L^2+\sigma^2\right)+3L_G^2}{2C_0\gamma L_G}+\frac{1}{C_0 \gamma}\left(2W_T+V_T\right).
	\end{align*}
	Choosing $\gamma$ as in \eqref{eq:gammamtnuchoicedspmzero}, and by \eqref{eq:relfwdynamicreg} we get \eqref{eq:rdspmboundzero}.
	\item After \eqref{eq:dynfwwt1breakdown}, choosing $\gamma$, and $m_t$ as in \eqref{eq:gammamtnuchoicedspmfirst} and following the same logic we get \eqref{eq:rdspmboundfirst}.
\end{enumerate}
\end{proof}
\begin{remark} 
Note that in Theorem~\ref{thm:fwdynamicfuncreg}, we require $\lbrace f_t \rbrace_{t=1}^T\in \mathcal{M}_T~\cap~\mathcal{D}_T$, unlike Theorem~\ref{th:dynamicexgrad} or in general \texttt{argmin}-type convex optimization problems, where the functions are generally required to belong to only one uncertainty set. The reason is due to the fact the merit function $w_{t}=f_{t}\left(x_{t},y_{t}^*\right)-f_{t}\left(x_{t}^*,y_{t}\right)$ is used as a suboptimality measure. At time $t$, if one had already known the points $\left(x_{t},y_{t}^*\right)$, and $\left(x_{t}^*,y_{t}\right)$ then the situation would be similar to online convex optimization. In particular, in this case, it would be enough to require $\lbrace f_t \rbrace_{t=1}^T\in \mathcal{M}_T$. But as we do not know $\left(x_{t},y_{t}^*\right)$, and $\left(x_{t}^*,y_{t}\right)$, we also require $\lbrace f_t \rbrace_{t=1}^T\in \mathcal{D}_T$.
\end{remark}

\subsection{Regret Analysis of Gradient Descent Ascent Algorithm}\label{sec:gdadynamic}
Thus far in this paper, we considered in extragradient and Frank-Wolfe algorithms for solving nonstationary saddle-point optimization problems of the form~\eqref{eq:nssaddle}. As discussed in Section~\ref{sec:mainalgos}, arguably the most natural algorithm for solving saddle-point optimization problem is the gradient descent ascent algorithm. Hence, it is worth exploring how well online or bandit versions of gradient descent algorithms performs for solving problems of the form~\eqref{eq:nssaddle}. In this section, we address this question concentrating on bounding dynamic regret. The algorithms is formally presented in Algorithm~\ref{alg:gdadynamica}. It turns out that it is not possible to obtain any meaningful bounds for the previous notions of dynamic regret (as in Definition~\ref{def:dsppath} or~\ref{def:dsp2}). We now define a weaker notions of dynamic regret which is suitable for analyzing Algorithm~\ref{alg:gdadynamica}. It is intriguing to explore other stronger notions of regret for which one could quantify the performance of Algorithm~\ref{alg:gdadynamica}, or prove impossibility results. 

\begin{definition} \label{def:dsp}
Let $\lbrace f_t \rbrace_{t=1}^T$ be a sequence of functions satisfying Assumption~\ref{as:strongconcon}. For this class of functions, with $(x_{t}^*,y_{t}^*)$ as defined in~\eqref{eq:nssaddle}, the Dynamic Saddle-Point (DSP) Regret is defined as 
\begin{align}
\mathfrak{R}_{DSP}=\expec{\abs{\sum_{t=1}^{T}\left(f_t\left(x_t,y_t\right)-f_t\left(x_t^*,y_t^*\right)\right)}}.
\end{align}
\end{definition}
\begin{algorithm}[t]\label{alg:gdadynamica}
	\caption{Bandit Gradient Descent Ascent Algorithm}
	\textbf{Input:} $x_1\in \mathbb{R}^{d_X}$, $y_1\in \mathbb{R}^{d_Y}$, $\eta_t>0$, $\nu_{X[Y]}>0$ 
	\begin{algorithmic}
	\State \textbf{for} $t=1,2,\cdots,T$ \textbf{do}	
	\State \textbf{Sample} $u_t^{x[y]} \sim N \left( 0,\bf{I_d}\right) $
	\State \textbf{Set} $\left[\bar{G}_t^x\left(x_t,y_t\right);\bar{G}_t^y\left(x_t,y_t\right)\right]=\texttt{gradest}\left(x_t,y_t,\nu_{X[Y]}\right)$ \Comment{Algorithm~\ref{alg:gradest}}
	\State \textbf{Update}
	\State $x_{t+1}=\mathcal{P}_\mathcal{X}\left( x_t-\eta_t \bar{G}^{x}_t\left(x_t,y_t \right)\right) $
	\State $y_{t+1}=\mathcal{P}_\mathcal{Y}\left( y_t+\eta_t \bar{G}^{y}_t\left(x_t,y_t \right)\right) $
	\State where $\mathcal{P}_\mathcal{X}\left( z\right)$ is the projection operator, i.e., $\mathcal{P}_\mathcal{X}\left( z\right)\vcentcolon=\argmin_{x \in \mathcal{X}}\norm{z-x}$
	\State \textbf{end for}
	\end{algorithmic}
\end{algorithm}
For the above mentioned notion of regret, we state the following result.
\begin{theorem}\label{th:absoutsum}
	Let $\left(x_t,y_t\right)$ be generated by Algorithm~\ref{alg:gdadynamica} for any sequence of functions $\lbrace f_t \rbrace_{t=1}^T \in \mathcal{M}_T$ for which Assumptions~\ref{as:lip} and \ref{as:lipgrad} hold. Then, we have:
	\begin{enumerate}[label=(\alph*)]
	    \item Under the availability of the stochastic zeroth-order oracle, choosing
	\begin{align}
	\eta_t=\eta=V_T^\frac{1}{4} \quad \nu_{X[Y]}=\frac{1}{\left(d_{X[Y]}+6\right)^\frac{3}{2}\sqrt{T}} \quad m_t^{X[Y]}=\left(d_{X[Y]}+6\right)T, \label{eq:dspetanuzero}
	\end{align}
	we obtain,
	\begin{align}
	\mathfrak{R}_{DSP}\leq\mathcal{O}\left(\left(1+\sigma^2\right)V_T^\frac{1}{4}+\sigma \sqrt{T}\right). \label{eq:dspzero}
	\end{align}
	Hence, the total number of calls to the stochastic zeroth-order oracle is $\mathcal{O}\left(\left(d_X+d_Y\right)T^2\right)$.
	\item Under the availability of the stochastic first-order oracle, choosing
	\begin{align}
	\eta_t=V_T^\frac{1}{4} \quad m_t=T, \label{eq:dspetanufirst}
	\end{align}
	we obtain,
	\begin{align}
	\mathfrak{R}_{DSP}\leq \mathcal{O}\left(V_T^\frac{1}{4}+\sigma \sqrt{T}\right). \label{eq:dspfirst}
	\end{align}	
	\end{enumerate}
\end{theorem}
\begin{proof}[Proof of Theorem~\ref{th:absoutsum}]
 Based on the non-expansiveness of the Euclidean projections and our boundedness assumption on $\mathcal{X}$, we have
	\begin{align*}
	&\norm{x_{t+1}-x_{t+1}^*}_2^2
	=\norm{x_{t+1}-x_{t}^*}_2^2+3D_\calX\|x_t^*-x_{t+1}^*\|\\
	=&  \norm{\mathcal{P}_\mathcal{X}\left( x_t-\eta_t \bar{G}^{x}_t\left(x_t,y_t \right)\right)-x_t^*   }_2^2+3D_\calX\|x_t^*-x_{t+1}^*\|\\
	\leq &  \norm{ x_t-\eta_t \bar{G}^{x}_t\left(x_t,y_t  \right)-x_t^*   }_2^2+3D_\calX\|x_t^*-x_{t+1}^*\|\\
	=& \norm{x_{t}-x_t^*}_2^2+\eta_t^2 \|\bar{G}^{x}_t\left(x_t,y_t\right)\|_2^2-2\eta_t \bar{G}^{x}_t\left(x_t,y_t\right)^\top\left( x_t-x_t^*\right)+3D_\calX\|x_t^*-x_{t+1}^*\|.
	\end{align*}
	Rearranging terms we then have
	\begin{align*}
	\bar{G}^{x}_t\left(x_t,y_t \right)^\top\left( x_t-x_t^*\right)\leq  \frac{1}{2\eta_t}\left(\norm{x_{t}-x_t^*}_2^2-\norm{x_{t+1}-x_{t+1}^*}_2^2 +\eta_t^2\norm{\bar{G}^{x}_t\left(x_t,y_t  \right)}_2^2+3D_\calX\|x_t^*-x_{t+1}^*\| \right). \numberthis \label{eq:notexpeckweagradx}
	\end{align*}
	Using the fact that $\sum_{t=1}^{T}\left(f_t\left(x_t^*,y_t\right)-f_t\left(x_t^*,y_t^*\right)\right)\leq 0$, and choosing $\eta_t$ as in \eqref{eq:dspetanuzero} we get
	\begin{align*}
	&\sum_{t=1}^{T}\left(f_t\left(x_t,y_t\right)-f_t\left(x_t^*,y_t^*\right)\right)\\
	=&\sum_{t=1}^{T}\left(f_t\left(x_t,y_t\right)-f_t\left(x_t^*,y_t\right)+f_t\left(x_t^*,y_t\right)-f_t\left(x_t^*,y_t^*\right)\right)\\
	\leq&\sum_{t=1}^{T}\left(\bar{G}^{x}_t\left(x_t,y_t  \right)^\top\left( x_t-x_t^*\right)+\|\bar{G}^{x}_t\left(x_t,y_t  \right)-\nabla_x f_t\left(x_t,y_t\right)\|\|x_t-x_t^*\|\right).
	\end{align*}
	Similarly, we have
	\begin{align*}
	    \sum_{t=1}^{T}\left(f_t\left(x_t^*,y_t^*\right)-f_t\left(x_t,y_t\right)\right)\leq \sum_{t=1}^{T}\left(\bar{G}^{y}_t\left(x_t,y_t  \right)^\top\left( y_t-y_t^*\right)+\|\bar{G}^{y}_t\left(x_t,y_t  \right)-\nabla_y f_t\left(x_t,y_t\right)\|\|y_t-y_t^*\|\right).
	\end{align*}
	Hence, we obtain
	\begin{align*}
	    \abs{\sum_{t=1}^{T}\left(f_t\left(x_t^*,y_t^*\right)-f_t\left(x_t,y_t\right)\right)}&~\leq  \sum_{t=1}^{T}\left(\bar{G}^{x}_t\left(x_t,y_t  \right)^\top\left( x_t-x_t^*\right)+\bar{G}^{y}_t\left(x_t,y_t  \right)^\top\left( y_t-y_t^*\right)\right)\\
	    &~+\sum_{t=1}^{T}\left(\|\bar{G}^{x}_t\left(x_t,y_t  \right)-\nabla_x f_t\left(x_t,y_t\right)\|D_\calX\right)\\
	    &~+\sum_{t=1}^{T} \left(\|\bar{G}^{y}_t\left(x_t,y_t  \right)-\nabla_y f_t\left(x_t,y_t\right)\|D_\calY\right).
	\end{align*}
	Hence, we obtain
	\begin{align*}
	    \expec{\sum_{t=1}^{T}\bar{G}^{x}_t\left(x_t,y_t  \right)^\top\left( x_t-x_t^*\right)}\leq & \frac{1}{2\eta}\left(\norm{x_{1}-x_1^*}_2^2-\norm{x_{T+1}-x_{T+1}^*}_2^2 +\eta^2\sum_{t=1}^{T}\expec{\norm{\bar{G}^{x}_t\left(x_t,y_t  \right)}_2^2}\right.\\
	    &\left.+3D_\calX\sum_{t=1}^{T}\|x_t^*-x_{t+1}^*\| \right).
	\end{align*}
	We now handle the zeroth-order and first-order setting separately. 
	\begin{enumerate}[label=(\alph*)]
    \item Choosing $\nu_{X[Y]}$, and $m_t^{X[Y]}$ as in \eqref{eq:dspetanuzero}, and using Lemma~\ref{lm:zerograd} we get
	\begin{align*}
	  \sum_{t=1}^{T}\expec{\norm{\bar{G}^{x}_t\left(x_t,y_t  \right)}_2^2}&\leq 3L^2+2\sigma^2,\\ \sum_{t=1}^{T}\expec{\|\bar{G}^{x[y]}_t\left(x_t,y_t  \right)-\nabla_{x[y]} f_t\left(x_t,y_t\right)\|}&\leq 2\left(L+\sigma+L_G\right)\sqrt{T}. \numberthis \label{eq:sqrtgraderror}
	\end{align*}
	Hence, we obtain
	\begin{align}
	    \expec{\sum_{t=1}^{T}\bar{G}^{x}_t\left(x_t,y_t  \right)^\top\left( x_t-x_t^*\right)}\leq & \frac{1}{2\eta}\left(D_\calX^2 +\eta^2\left(3L^2+2\sigma^2\right)+3D_\calX\sum_{t=1}^{T}\|x_t^*-x_{t+1}^*\| \right). \label{eq:gbartxta} 
	\end{align}
	Similarly, we also have
	\begin{align}
	    \expec{\sum_{t=1}^{T}\bar{G}^{y}_t\left(x_t,y_t  \right)^\top\left( y_t-y_t^*\right)}\leq & \frac{1}{2\eta}\left(D_\calY^2 +\eta^2\left(3L^2+2\sigma^2\right)+3D_\calY\sum_{t=1}^{T}\|y_t^*-y_{t+1}^*\| \right). \label{eq:gbartxtb}
	\end{align}
	Combining, \eqref{eq:sqrtgraderror}, \eqref{eq:gbartxta}, and \eqref{eq:gbartxtb}, we get
	\begin{align*}
	\expec{\abs{\sum_{t=1}^{T}\left(f_t\left(x_t,y_t\right)-f_t\left(x_t^*,y_t^*\right)\right)}}&\leq  \frac{1}{2\eta}\left(D_\calX^2+D_\calY^2+\eta^2\left(6L^2+4\sigma^2\right)+3\left(D_\calX+D_\calY\right)\sqrt{2V_T}\right)\\&~~~+2\left(D_\calX+D_\calY\right)\left(L+L_G+\sigma\right)\sqrt{T}.
	\end{align*}
	Choosing $\eta$ as in \eqref{eq:dspetanuzero} we get,
	\begin{align*}
	\mathfrak{R}_{DSP}\leq  \mathcal{O}\left(\left(1+\sigma^2\right)V_T^\frac{1}{4}+\sigma \sqrt{T}\right)
	\end{align*}
	\item The proof for part (b) is similar to part (a). Choosing $\eta_t$, $\nu$, and $m_t$ as in \eqref{eq:dspetanufirst} we get \eqref{eq:dspfirst}.
\end{enumerate}

\end{proof}
\begin{remark}
Note that $\mathfrak{R}_{DSP}$ is a weaker notion of regret than $\mathfrak{R}_{DSPF}$ as we trivially have $$\mathfrak{R}_{DSP}=\expec{\abs{\sum_{t=1}^{T}\left(f_t\left(x_t,y_t\right)-f_t\left(x_t^*,y_t^*\right)\right)}}\leq \expec{\sum_{t=1}^{T}\abs{f_t\left(x_t,y_t\right)-f_t\left(x_t^*,y_t^*\right)}}=\mathfrak{R}_{DSPF}.$$
Hence, the obtained sub-linear regret bounds on  $\mathfrak{R}_{DSP}$ in Theorem~\ref{th:absoutsum} have no consequence for $\mathfrak{R}_{DSPF}$. 
\end{remark}

%% file: appendix.tex
\section{Proof of Theorem~\ref{theo:offl_nonadapt}} \label{sec:proofofzofw}
In this section, we prove Theorem~\ref{theo:offl_nonadapt}. In order to do so, we require a few sub-results which we state and prove below.
\begin{lemma}
\label{lm: zeroOrder_helper1}
Under assumption~\ref{as:lipgrad}, the following inequalities hold:
\begin{align*}
    &\gamma_k\inner{\derivx{x_{k-1},y_{k-1}},s_{k}^x-x_{k-1}} \leq -\gamma_k \wh{g}_{k-1}^x + \frac{L_G\gamma^2}{2} D^2_\calX + \frac{1}{2L_G}\norm{\Delta_k^x}^2, \\
    &\gamma_k\inner{-\derivy{x_{k-1},y_{k-1}},s_{k}^y-y_{k-1}} \leq -\gamma_k \wh{g}_{k-1}^y + \frac{L_G\gamma^2}{2} D^2_\calY + \frac{1}{2L_G}\norm{\Delta_k^y}^2,
\end{align*}
where $\norm{\Delta_k^x}:= \bar{G}_k^{x}\left(x_k,y_k\right)-\nabla_{x}f\left(x_k,y_k\right)$ and $\norm{\Delta_k^y}:=\bar{G}_k^{y}\left(x_k,y_k\right)-\nabla_{y}f\left(x_k,y_k\right)$. 
\end{lemma}
\begin{proof}[Proof of Lemma~\ref{lm: zeroOrder_helper1}]
The proof follows by the fact that
    \begin{align*}
    &\gamma_k\inner{\derivx{x_{k-1},y_{k-1}},s_{k}^x-x_{k-1}}\\
    =~&\gamma_k\inner{\derivx{x_{k-1},y_{k-1}},\wh{s}_{k}^x-x_{k-1}} + \gamma_k\inner{\derivx{x_{k-1},y_{k-1}},s_k^x - \wh{s}_{k}^x}\\
    \leq~&  -\gamma_k\wh{g}_{k-1}^x + \gamma_k\inner{\Delta_k^x,s_k^x - \wh{s}_{k}^x} \\
    \leq~&  - \gamma_k\wh{g}_{k-1}^x + \frac{L_G\gamma_k^2}{2} D^2_\calX + \frac{1}{2L_G}\norm{\Delta_k^x}^2,
    \end{align*}
where, the first inequality follows from the observation that $\inner{\Bar{G}^{x}_k,s_k^x - u} \leq 0$ due to the optimality condition of $s_k^x$. The second inequality of Lemma~\ref{lm: zeroOrder_helper1} follows similarly.
\end{proof}
\begin{lemma} 
\label{lm:est_fwgap_error}
The gap between the true optimality error and estimated optimality error are bounded as
\begin{align*}
 \abs{\wh{g}_k^x- g_k^x} \leq D_\calX\norm{\Delta_k^x}, \quad
    \abs{\wh{g}_k^y- g_k^y} \leq D_\calY\norm{\Delta_k^y}, \quad
    \abs{\wh{g}_k- g_k} \leq (D_\calX\norm{\Delta_k^x}+D_\calY\norm{\Delta_k^y}).
\end{align*}
\end{lemma}

\begin{proof}[Proof of Lemma~\ref{lm:est_fwgap_error}]
First note that, we have
\begin{align*}
    \wh{g}_k^x- g_k^x 
    &= -\inner{\derivx{x_{k},y_k},\wh{s}_{k+1}^x-x_k} + \inner{\Bar{G}_{k+1}^{x}, s_{k+1}^x - x_k}\\
    &= -\inner{\derivx{x_{k},y_k},\wh{s}_{k+1}^x-x_k} + \inner{\Bar{G}_{k+1}^{x}, \wh{s}_{k+1}^x - x_k} + \inner{\Bar{G}_{k+1}^{x}, s_{k+1}^x -\wh{s}_{k+1}^x}\\
    &\leq -\inner{\derivx{x_{k},y_k},\wh{s}_{k+1}^x-x_k} + \inner{\Bar{G}_{k+1}^{x}, \wh{s}_{k+1}^x - x_k}\quad \text{ (By the optimality of $s_{k+1}^x$ )}\\
    &= \inner{ \Delta_k^x,\wh{s}_{k+1}^x-x_k}
    \leq \norm{\Delta_k^x} D_\calX.
\end{align*}
Similarly, we have
\begin{align*}
    g_k^x - \wh{g}_k^x
    &= \inner{\derivx{x_{k},y_k},\wh{s}_{k+1}^x-x_k} - \inner{\Bar{G}_{k+1}^{x}, s_{k+1}^x - x_k}\\
    &= \inner{\derivx{x_{k},y_k},s_{k+1}^x-x_k} + \inner{\derivx{x_{k},y_k}, \wh{s}_{k+1}^x - s_{k+1}} - \inner{\Bar{G}_{k+1}^{x}, s_{k+1}^x - x_k}\\
    &\leq \inner{\derivx{x_{k},y_k},s_{k+1}^x-x_k}- \inner{\Bar{G}_{k+1}^{x}, s_{k+1}^x - x_k}\quad \text{ (By the optimality of $\wh{s}_{k+1}^x$ )}\\
    &= \inner{ -\Delta_k^x,s_{k+1}^x-x_k}
    \leq \norm{\Delta_k^x}D_\calX.
\end{align*}
Hence the upper bound for $\abs{\wh{g}_k^x- g_k^x}$ is proved. By using a similar approach, the upper bound claim for $|\wh{g}_k^y- g_k^y|$ could be proved. Hence by the Definition~\ref{def:fwgaps} and triangle inequality, the upper bound claim for $\abs{\wh{g}_k- g_k}$ follows.
\vgap
\end{proof}
\begin{proof}[Proof of Theorem~\ref{theo:offl_nonadapt}]
First note that by Assumption~\ref{as:lipgrad},
\begin{align*}
    f(x_{k},y^*) 
    &\leq f(x_{k-1},y^*) + \inner{\derivx{x_{k-1},y^*},x_{k}-x_{k-1}} + \frac{L_{GX}}{2}\norm{x_{k}-x_{k-1}}^2\\
    & = f(x_{k-1},y^*) + \gamma_k\inner{\derivx{x_{k-1},y^*},s_{k}-x_{k-1}} + \frac{L_{GX}\gamma_k^2}{2}\norm{s_{k}-x_{k-1}}^2.
    \end{align*}
Hence, by Assumption~\ref{as:crosslipgrad} and Lemma~\ref{lm: zeroOrder_helper1},
    \begin{align*}
   f(x_{k},y^*)  &\leq f(x_{k-1},y^*) + \gamma_k\inner{\derivx{x_{k-1},y_{k-1}},s_{k}-x_{k-1}} + \gamma_k D_\calX L_{XY} \norm{y_{k-1}-y^*}  + \frac{L_{GX}\gamma^2_k}{2}D^2_\calX \\ 
    & \leq f(x_{k-1},y^*) - \gamma_k\wh{g}_{k-1}^x + L_G\gamma^2 D^2_\calX + \gamma_k D_\calX L_{XY} \norm{y^* - y_{k-1}}  + \frac{1}{2L_G}\norm{\Delta_k^x}^2. \numberthis\label{eq:meritineq1}
\end{align*}
Similarly we have,
\begin{align*}
    -f(x^*,y_k) \leq -f(x^*,y_{k-1}) - \gamma_k\wh{g}_{k-1}^y + L_G\gamma^2 D^2_\calY + \gamma_k D_\calY L_{YX} \norm{x^* - x_{k-1}}  + \frac{1}{2L_G}\norm{\Delta_k^y}^2. \numberthis\label{eq:meritineq2}
\end{align*}
Adding \eqref{eq:meritineq1}, and \eqref{eq:meritineq2}, and using Assumption~\ref{as:strongconcon} we have
\begin{align*}
    w_{k} - w_{k-1} 
    &\leq -\gamma_k\wh{g}_{k-1} + \gamma_k (D_\calX L_{XY} \norm{y^* - y_{k-1}} + D_\calY L_{YX} \norm{x^* - x_{k-1}}) + L_G\gamma^2 (D^2_\calX + D^2_\calY)\\
    & \quad + \frac{1}{2L_G}(\norm{\Delta_k^x}^2 + \norm{\Delta_k^y}^2)\\
    &\leq -\gamma_k\wh{g}_{k-1} + \gamma_k \left(D_\calX L_{XY} \sqrt{\frac{2w_{k-1}^y}{\mu_y}} + D_\calY L_{YX} \sqrt{\frac{2w_{k-1}^x}{\mu_x}}\right) + L_G\gamma^2 (D^2_\calX + D^2_\calY) \\ &\quad+ \frac{1}{2L_G}(\norm{\Delta_k^x}^2 + \norm{\Delta_k^y}^2)\\
    &\leq -\gamma_k\wh{g}_{k-1} + 2\gamma_k \max\left\{ \frac{D_\calX L_{XY}}{\sqrt{\mu_\calY}},\frac{D_\calY L_{YX}}{\sqrt{\mu_\calX}}\right\}\sqrt{w_{k-1}} + L_G \gamma_k^2 (D^2_\calX + D^2_\calY) \\ &\quad + \frac{1}{2L_G}(\norm{\Delta_k^x}^2 + \norm{\Delta_k^y}^2)
\end{align*}
So by Lemma 19 in \cite{gidel2017frank}, we get 
\begin{align}
    w_k \leq w_{k-1} - C_0\gamma_k\wh{g}_{k-1} + \gamma^2_k C_1 + \frac{1}{2L_G}(\norm{\Delta_k^x}^2 + \norm{\Delta_k^y}^2) \label{eq:standard_form} 
\end{align}
Now we are ready to prove the two results. First we will prove the result (a) with non-adaptive step size. Because $w_k \leq \wh{g}_k$, \eqref{eq:standard_form} can be rewritten as 
\begin{align*}
w_{k} \leq \left(1-\frac{C_0\gamma_k}{2}\right)w_{k-1} - \frac{C_0\gamma_k}{2}\wh{g}_{k-1} + \gamma^2_kC_1 + \frac{1}{2L_G}\left(\norm{\Delta_k^x}^2 + \norm{\Delta_k^y}^2\right).
\end{align*}
By taking expectation on the both side, dividing by $\Gamma_k$, summing up and doing a telescoping argument, we have 
\begin{align*}
    C_0\sum_{k=1}^{N}\frac{\gamma_k}{2\Gamma_k}\E[\wh{g}_k] + \frac{\E[w_{k}]}{\Gamma_{k}}
    &\leq \E[w_0] - \frac{C_0}{2}\sum_{k=1}^{N}\frac{\gamma_k}{\Gamma_k}\E[w_{k-1}] + C_1\sum_{k=1}^{N}\frac{\gamma_k^2}{\Gamma_k} +\frac{1}{2L_G}\sum_{k=1}^{N}\frac{\E[\norm{\Delta_k^x}^2 + \norm{\Delta_k^y}^2]}{\Gamma_k}\\
    &\leq \E[w_0] + C_1\sum_{k=1}^{N}\frac{\gamma_k^2}{\Gamma_k}+\frac{1}{2L_G}\sum_{k=1}^{N}\frac{\E[\norm{\Delta_k^x} + \norm{\Delta_k^y}]}{\Gamma_k}.~~~\text{(as $C_0 > 0$)}
\end{align*}
\newline
Then by Lemma~\ref{lm:zerograd}, the choice of $\Gamma_0 = 1$ and the fact that $\sum_{k=1}^{N} \frac{\gamma_k\Gamma_{T}}{2\Gamma_k(1-\Gamma_{T})} = 1$, we have
\begin{align*}
    \E[w_N] + \E[\wh{g}_R] 
    &= \Gamma_{N}\left(\frac{\E[w_{N}]}{\Gamma_{N}}\right) + \frac{\Gamma_{N}}{C_0(1-\Gamma_{N})} \left(C_0\sum_{k=1}^{N}\frac{\gamma_k}{2\Gamma_k}\E[\wh{g}_k]\right)\\
    &\leq \frac{\Gamma_{N}}{C_0(1-\Gamma_{N})}\left(\frac{\E[w_{N-1}]}{\Gamma_{N-1}}\right) + \frac{\Gamma_{N}}{C_0(1-\Gamma_{N})}\left(C_0\sum_{k=1}^{N}\frac{\Gamma_k}{2\Gamma_k}\E[\wh{g}_k]\right)\\
    &\leq \frac{\Gamma_{N}}{C_0(1-\Gamma_{N})} \left[ w_0 + C_1\sum_{k=1}^{N}\frac{\gamma_k^2}{\Gamma_k} +\frac{1}{2L_G}\sum_{k=1}^{N}\frac{\E[\norm{\Delta_k^x} + \norm{\Delta_k^y}]}{\Gamma_k}  \right]\\
    & \leq \frac{\Gamma_{N}}{C_0(1-\Gamma_{N})} \left[ w_0 + C_1\sum_{k=1}^{N}\frac{\gamma_k^2}{\Gamma_k} + \left(\frac{2L_G}{T^2}+\frac{3L_G}{4T^2}\right)\left(B_x^{L\sigma}+B_y^{L\sigma}\right)\sum_{k=1}^{N} \frac{1}{\Gamma_k}\right] \\
    & = \frac{\Gamma_{N}}{C_0(1-\Gamma_{N})}w_0 + \frac{C_1\Gamma_{N-1}}{C_0(1-\Gamma_{N-1})}\sum_{k=1}^{N}\frac{\gamma_k^2}{\Gamma_k} + \frac{\left(\frac{2L_G}{N^2}+\frac{3L_G}{4N^2}\right)\left(B_x^{L\sigma}+B_y^{L\sigma}\right)}{C_0(1-\Gamma_{N-1})} \left(\Gamma_{N-1}  \sum_{k=1}^{N} \frac{1}{\Gamma_k}\right).
\end{align*}
The first inequality comes from the fact that $1-\Gamma_{N-1} < 1$ and $C_0 < 1$. Now, it is easy to verify that the following inequalities hold:
\begin{align*}
    \Gamma_k \leq \frac{60}{(k+3)(k+4)(k+5)}\qquad \frac{1}{1-\Gamma_N} &\leq 2 \qquad \Gamma_{N-1}  \sum_{k=1}^{N} \frac{1}{\Gamma_k} \leq N\\
    \sum_{k=1}^{N}\frac{\gamma_k^2}{\Gamma_k}\leq \sum_{k=1}^{N}\frac{3(k+3)}{5}& \leq \frac{3N(N+7)}{10}.
\end{align*}
Based on the above set of inequalities, we then have
\begin{align*}
    \E[w_k] + \E[\wh{g}_R]  \leq \frac{120 w_0}{(N+3)^3} + \frac{36L_GC_1}{C_0(N+5)} + \frac{11\left(\sqrt{(L_{GX})^2+\sigma^2}+\sqrt{(L_{GY})^2+\sigma^2}\right)}{2NC_0} : = \epsilon.
\end{align*}
Therefore we have $N = \order\left({1}/{\epsilon}\right)$ and hence the total number of call to stochastic zeroth-order oracle is 
\begin{align*}
    \sum_{k=1}^N (m_k^X + m_k^Y) 
    = \sum_{k=1}^N N^2 [B^{L\sigma}_X(d_X+5) + B^{L\sigma}_Y(d_Y+5)]
     = \order((d_X+d_Y)T^3) = \order(\frac{d_X+d_Y}{\epsilon^3}).
\end{align*}
Next, we prove the result (b) with an adaptive step-size choice for $\gamma_k$.
Using Lemma~\ref{lm:est_fwgap_error}, \eqref{eq:standard_form}, $C_0<1$ we get,
\begin{align*}
    w_k
    &\leq w_{k-1}-C_0\gamma_k g_{k-1} + \gamma^2_k C_1 + \frac{1}{2L_G}(\norm{\Delta_k^x}^2 + \norm{\Delta_k^y}^2) + C_0\gamma_k(D_\calX\norm{\Delta_k^x}+D_\calY\norm{\Delta_k^y}) \\
    &\leq w_{k-1}-C_0\gamma_k g_{k-1} + \gamma^2_k C_1 + \frac{1}{2L_G}(\norm{\Delta_k^x}^2 + \norm{\Delta_k^y}^2) + \gamma_k(D_\calX\norm{\Delta_k^x}+D_\calY\norm{\Delta_k^y}) \\    
    & \leq w_{k-1}-C_0\gamma_k g_{k-1} + 2\gamma^2_k C_1 + \frac{1}{L_G}(\norm{\Delta_k^x}^2 + \norm{\Delta_k^y}^2).
\end{align*}
Note that in Algorithm~\ref{algo:offline_SPFW} we set $\gamma_k = \min\{1, \frac{C_0}{4C_1}g_{k-1} \}$. So when $\frac{C_0}{4C_1}g_{k-1} < 1$, $\gamma_k =\frac{C_0}{4C_1}g_{k-1}$.  By Lemma~\ref{lm:est_fwgap_error} and Lemma 19 in \cite{gidel2017frank} we then obtain
\begin{align*}
    w_k
    &\leq w_{k-1}- \frac{C^2_0}{8C_1}(g_{k-1})^2 + \frac{1}{L_G}(\norm{\Delta_k^x}^2 + \norm{\Delta_k^y}^2)\\
    &\leq w_{k-1}- \frac{C^2_0}{8C_1}\left(\frac{1}{2}\hat{g}_{k-1}^2-\norm{\hat{g}_{k-1}-g_{k-1}}^2\right) + \frac{1}{L_G}(\norm{\Delta_k^x}^2 + \norm{\Delta_k^y}^2)\\
    &\leq w_{k-1}- \frac{C^2_0}{16C_1}\wh{g}_{k-1}^2+ \left(\frac{1}{4C_1}D^2_\calX + \frac{1}{L_G}\right)\norm{\Delta_k^x}^2 +\left(\frac{1}{4C_1}D^2_\calY + \frac{1}{L_G}\right)\norm{\Delta_k^y}^2 \\
    &\leq \left(1- \frac{C^2_0\delta^2_\mu}{8C_1}\right)w_{k-1}+\underbrace{ \left(\frac{1}{4C_1}D^2_\calX + \frac{1}{L_G}\right)\norm{\Delta_k^x}^2 +\left(\frac{1}{4C_1}D^2_\calY + \frac{1}{L_G}\right)\norm{\Delta_k^y}^2}_{T_2}.
\end{align*}
In the other case, when $\frac{C_0}{4C_1}g_{k-1} \geq 1$, $\gamma_k =1$, and we obtain,
\begin{align*}
    w_k
    &\leq w_{k-1}- C_0 g_{k-1} + \frac{C_0}{2}  g_{k-1} +  \frac{1}{L_G}(\norm{\Delta_k^x}^2 + \norm{\Delta_k^y}^2)\\
    &= w_{k-1}- \frac{C_0}{2}  g_{k-1} + \frac{1}{L_G}(\norm{\Delta_k^x}^2 + \norm{\Delta_k^y}^2)\\
    & \leq w_{k-1}- \frac{C_0}{2}  \wh{g}_{k-1} + \frac{C_0}{2} (D_\calX\sqrt{\norm{\Delta_k^x}^2}+D_\calY\sqrt{\norm{\Delta_k^y}^2}) +\frac{1}{L_G}(\norm{\Delta_k^x}^2 + \norm{\Delta_k^y}^2)\\
    & \leq (1-\frac{C_0}{2} ) w_{k-1} +\underbrace{ \frac{C_0}{2} (D_\calX\sqrt{\norm{\Delta_k^x}^2}+D_\calY\sqrt{\norm{\Delta_k^y}^2}) +\frac{1}{L_G}(\norm{\Delta_k^x}^2 + \norm{\Delta_k^y}^2)}_{T_3},
\end{align*}
where the last inequality follows by the fact that $w_k\leq \wh{g}_k$. Hence, by defining $\rho = 1- \min\left\{\frac{C^2_0\delta^2_\mu}{8C_1}, \frac{C_0}{2} \right\}$ for convenience, we can get
\begin{align*}
    &w_k \leq (1-\rho)w_{k-1} + T_2,~\text{when}~\gamma = \frac{C_0}{4C_1},\\
    &w_k \leq (1-\rho)w_{k-1} + T_3,~\text{when}~\gamma = 1.
\end{align*}
Now rearrange the inequality and taking expectation on the both side, and by Lemma~\ref{lm:zerograd}, we get
\begin{align*}
    &\E[w_k] - \frac{1}{\rho}\max\left\{\E[T_2],\E[T_3]\right\}
    \leq  (1-\rho)\left(\E[w_{k-1}] - \frac{1}{\rho}\max\left\{\E[T_2],\E[T_3]\right\}\right). 
\end{align*}
Therefore, we have
\begin{align*}
    \E[w_k] \leq (1-\rho)^T (\E[w_0] - \frac{1}{\rho}\max\{\E[T_2],\E[T_3]\}) + \frac{1}{\rho}\max\{\E[T_2],\E[T_3]\}.
\end{align*}
This implies the iterates decreases geometrically to reach point that is $\left(\frac{1}{\rho}\max\{\E[T_2],\E[T_3]\}\right)$-close to the saddle point. Notice that $\rho$ is a problem dependent constant and the magnitude of $T_2$ and $T_3$ depends on the $\Delta_k^{x[y]}$, which is the gradient estimation error. Hence by choosing the values of $m_k,v_k$ in the gradient estimation oracle as stated in~\eqref{eq:gammamtnutPRchoice_adp}, we get an $\epsilon$-optimal saddle saddle point, with the stated number of calls to the stochastic zeroth-order oracle and linear sub-problems.

\end{proof}